\documentclass[10pt]{article}
\usepackage[a4paper,margin=19mm]{geometry}
\usepackage{amsmath,amssymb,bm,mathtools,amsthm}
\usepackage{graphicx,booktabs,array,multirow,microtype}
\graphicspath{{./}}
\usepackage{caption,subcaption,float,placeins}
\newfloat{algorithm}{tbp}{loa}
\floatname{algorithm}{Algorithm}
\newenvironment{algorithmic}[1][]{\begin{enumerate}[leftmargin=*,label=\arabic*.,nosep]}{\end{enumerate}}
\newcommand{\State}{\item}

\usepackage{xcolor,enumitem,setspace}
\usepackage[bookmarks=false]{hyperref}

\hypersetup{
 colorlinks=true,linkcolor=black,citecolor=black,urlcolor=black,
 pdftitle={Quadratic Stiff Damping and Full-Step Asymptotic Preservation in a Two-Stage Fourth-Order Finite-Volume Method for Cattaneo Relaxation},
 pdfauthor={Zhixin Huo},
 pdfkeywords={two-derivative time integration, exact discrete trajectory derivative, quadratic stiff damping, finite-volume method, Cattaneo relaxation, asymptotic preservation}
}
\newcommand{\UU}{\bm U}\newcommand{\LL}{\mathcal L}\newcommand{\GG}{\mathcal G}
\newcommand{\FF}{\bm F}\newcommand{\SSS}{\bm S}\newcommand{\qq}{\bm q}
\newcommand{\Dt}{\Delta t}\newcommand{\dd}{\mathrm d}\newcommand{\eps}{\varepsilon}

\newtheorem{theorem}{Theorem}
\newtheorem{lemma}{Lemma}
\newtheorem{proposition}{Proposition}
\newtheorem{corollary}{Corollary}
\newtheorem{remark}{Remark}
\newtheorem{definition}{Definition}
\newenvironment{revision}{\begingroup\color{black}}{\endgroup}
\newenvironment{newrevision}{\begingroup\color{red!75!black}}{\endgroup}

\title{\color{black}Quadratic Stiff Damping and Full-Step Asymptotic Preservation\\in a Two-Stage Fourth-Order Finite-Volume Method for Cattaneo Relaxation}
\author{\begin{tabular}{c}
Zhixin Huo$^{1,2,*}$
and
Yangnan Su$^{1}$
\\[3pt]
\small $^1$School of Mathematics and Information Science, Henan Polytechnic University\\
\small Jiaozuo 454003, China\\[2pt]
\small $^2$School of Mechatronical Engineering, Beijing Institute of Technology\\
\small Beijing 100081, China\\[2pt]
\small $^*$Corresponding author: \texttt{zhixinhuo@hpu.edu.cn}
\end{tabular}}
\date{}

\begin{document}
\maketitle

\begin{abstract}
\begin{revision}
We couple a previously analyzed L-stable sequential two-stage fourth-order time formula to a specified fifth-order finite-volume residual for the periodic Goldstein--Taylor/Cattaneo equation in one and two dimensions. The required second time derivative is chosen as the exact derivative along the semi-discrete trajectory, $\mathcal G_h(U)=D\mathcal L_h(U)\mathcal L_h(U)$. For our linear reconstruction and flux, applying the same conservative face operator to $V=\mathcal L_h(U)$ yields this derivative exactly. A central face flux produces a skew-adjoint divergence--gradient pair, allowing a fixed-grid full-step asymptotic-preserving (AP) theorem for the implemented operator; \begin{newrevision}for the specified central-flux discrete heat target, a wave-speed Rusanov flux has the same derivative closure but fails that target. This statement is not a claim about Rusanov-type discretizations in general.\end{newrevision} Within the previously established L-stable parameter interval, $C_q=5/183$ is the unique fixed coefficient \begin{newrevision}within the admissible family, in the sense of the asymptotic damping exponent,\end{newrevision} that cancels the leading inverse-stiffness amplification, giving $R(-x)=O(x^{-2})$ rather than the generic $O(x^{-1})$ as $x=\Delta t/\delta\to\infty$. Independent high-precision scalar calculations and the actual two-dimensional finite-volume transverse mode verify the exponent up to $x=10^5$; a mixed slow--fast state shows the resulting flux-error advantage for specified data. \begin{newrevision}This damping statement does not optimize finite-stiffness accuracy or computational efficiency. The AP statements below concern the fixed-grid linear periodic problem with exact stage solutions; no mesh-uniform or nonlinear AP result is claimed.\end{newrevision}
\end{revision}
\end{abstract}

\noindent\textbf{Keywords:} two-derivative integration; exact discrete trajectory derivative; quadratic stiff damping; finite-volume method; Cattaneo equation; asymptotic preservation

\section{Introduction}
\begin{revision}
We begin with the one-dimensional Goldstein--Taylor/Cattaneo equation on the periodic interval $x\in[0,2\pi]$:
\begin{equation}
 u_t+q_x=0,\qquad \delta q_t+a u_x=-q,\qquad a>0,\quad\delta>0,
\label{eq:intro-model}
\end{equation}
with periodic initial data $u(0,x)=u_0(x)$ and $q(0,x)=q_0(x)$. The exact limiting constraint obtained by setting $\delta=0$ is $q=-a u_x$, and the limit equation is $u_t=a u_{xx}$. At positive $\delta$ the identity is $q=-a u_x-\delta q_t$; an $O(\delta)$ remainder requires $q_t$ to remain bounded and can fail in an unprepared initial layer. We will distinguish this limit from fourth-order temporal accuracy at a fixed positive $\delta$.

To state the numerical question before using the notation, let $\mathcal L_h:\mathbb R^{N_{\rm dof}}\to\mathbb R^{N_{\rm dof}}$ be a semi-discrete conservative residual. Here $N_{\rm dof}$ is the total number of scalar unknowns, $Z$ is a stage state in this space, and $D\mathcal L_h(Z)$ means its Fr\'echet derivative (Jacobian) at $Z$; it is not the spatial divergence $D_h$. The exact derivative of the residual along its own trajectory is
\begin{equation}
 \mathcal G_h^{\rm tr}(Z)=D\mathcal L_h(Z)\,\mathcal L_h(Z).
\label{eq:intro-tr}
\end{equation}
The time formula of \cite{Huo2026} uses this second derivative. Its coefficients, fourth-order consistency, stability function, and L-stability are analyzed there. We state the formula for completeness but do not present those temporal facts as new contributions.

\textbf{Prior results and the present contribution.} \begin{newrevision}Sch\"utz, Seal, and Jaust \cite{SchutzSealJaust2017} identified the mismatch between a physical Cauchy--Kowalevski (CK) PDE expansion and the derivative of a method-of-lines (MOL) discretization, and constructed higher derivatives from the discrete MOL operator for discontinuous Galerkin methods. Huo \cite{Huo2026} analyzed the two-stage temporal coefficients, fourth-order consistency, and L-stability used here. These are the prior principles and temporal results on which this paper builds. For the specified \emph{linear constant-coefficient} equation, our contribution is to realize the trajectory derivative through conservative fifth-order finite-volume face assembly, verify the central flux's adjoint spatial pairing, and prove full-step AP for that \emph{same implemented} residual in one and two dimensions, with direct numerical audits. Within the inherited temporal family, we also characterize the unique fixed coefficient $C_q=5/183$ with quadratic inverse-stiffness decay and test its effect on two-dimensional transverse and mixed slow--fast fields. This parameter characterization concerns a property of the existing stability function, rather than a new temporal formula. In every stage, $D\mathcal L_h(U)\mathcal L_h(U)$ is realized by applying the same linear face operator to $V=\mathcal L_h(U)$. The resulting linear closure identity is an implementation condition, not a new general differentiation principle; the AP proof additionally requires the spatial pairing. No physical local ADER/CK predictor is assumed, and no mesh-uniform or nonlinear AP result is claimed.\end{newrevision}

Earlier explicit two-stage Lax--Wendroff and multiderivative finite-volume methods \cite{LiDu2016,SealGucluChristlieb2014} motivate the compact stage structure, while implicit multiderivative schemes with AP properties have been developed under different stability and splitting assumptions \cite{GottliebGrantHuShu2022}. AP and IMEX methods for diffusive relaxation provide a broader comparison \cite{Jin1999,PareschiRusso2005,BoscarinoPareschiRusso2013}. \begin{newrevision}Our narrower question is whether this particular sequential implicit time formula and its \emph{implemented} conservative FV operator share both exact trajectory closure and the required fixed-grid diffusion limit.\end{newrevision}

The compatible pair satisfies $D_h=-G_h^*$ in its discrete inner products. We prove that the central finite-volume branch has precisely this structure, so the theorem concerns both an abstract fixed-grid pair and the stated implementation. Section~\ref{sec:two-dimensional-main} extends the same linear equation to two spatial dimensions with the same reconstructed face stencil, proves AP for the paired central branch, and isolates a Rusanov checkerboard counterexample. The independently compiled supplement retains auxiliary one-dimensional diagnostics, the Prothero--Robinson ODE, Jin--Xin, and a separate centered-difference two-dimensional calculation; none supplies a nonlinear or mesh-uniform AP theorem.
\end{revision}

\begin{table}[H]
\centering
\caption{\textcolor{black}{Guide to the main article after defining its model and trajectory derivative.}}
\label{tab:claim-map}
\small
\begin{tabular}{>{\raggedright\arraybackslash}p{.27\linewidth}>{\raggedright\arraybackslash}p{.32\linewidth}>{\raggedright\arraybackslash}p{.31\linewidth}}
\toprule
Question & analytical result & direct numerical check\\
\midrule
Does the face derivative equal \eqref{eq:intro-tr}? & linear flux/reconstruction identity and perturbation bound & one- and two-dimensional finite-volume audits and temporal order\\
Does the complete step reach diffusion? & fixed-grid $O(\delta)$ theorem for a paired flux & one- and two-dimensional central/Rusanov FV comparisons\\
Which fixed parameter damps deep relaxation fastest? & unique $C_q$ and its finite-$x$ caveat & two-dimensional fast-mode and mixed-field sweeps\\
\bottomrule
\end{tabular}
\end{table}

\begin{revision}
The analysis follows the implemented operator. Exact face-derivative closure supplies the time formula's required trajectory derivative; the central flux additionally supplies the paired spatial operators needed for the complete-step AP limit. The Rusanov comparison shows why the first property alone is insufficient for the specified central-flux discrete heat target. In two dimensions, a transverse fast-flux mode isolates the $x^{-2}$ decay at $C_q$, and a mixed slow--fast state tests its effect alongside diffusion. Separate initial-layer estimates delimit uniform temporal accuracy. Exact semi-discrete and, where stated, positive-$\delta$ PDE references keep time, space, and AP errors distinct. We make no measured speedup claim; auxiliary models and different stencils appear in the supplement.
\end{revision}

\section{The one-dimensional fully stiff relaxation model}
\subsection{Goldstein--Taylor/Cattaneo equation and compatible operators}
\begin{revision}
The one-dimensional model in Sections 2--9 is \eqref{eq:intro-model} on its stated periodic domain; the separate boundary experiments prescribe flux data where used. Cattaneo heat conduction takes $u=T$, $q$ as heat flux, $\delta=\tau$, and $a=\kappa$ \cite{Cattaneo1958,Vernotte1958}. Goldstein--Taylor variables correspond to $u=\rho$, $q=j$, $\delta=\eps^2/2$, and $a=1/2$; this is the same one-dimensional model after a change of variables.

At $\delta=0$ its equilibrium relation is exactly $q=-a u_x$. For $\delta>0$ the exact identity is $q=-a u_x-\delta q_t$. Only when $q_t$ is bounded uniformly on the interval considered may one write $q=-a u_x+O(\delta)$ and hence
\begin{equation}
 u_t=a u_{xx}+O(\delta).
\label{eq:diffusion-limit}
\end{equation}
The numerical challenge is to recover the heat limit at fixed spatial grid and time step, without resolving the wave speed $O(\delta^{-1/2})$ or a time scale $O(\delta)$.

Let $X_h$ be the finite-dimensional scalar space and $Q_h$ the flux space. For a periodic uniform grid one may take $X_h=Q_h=\mathbb R^N$, with $N$ cells, cell-centered $u_i$, face-centered $q_{i+1/2}$, and inner products $\langle v,w\rangle_{X_h}=h\sum_i v_iw_i$ and $\langle r,s\rangle_{Q_h}=h\sum_i r_{i+1/2}s_{i+1/2}$. Their norms are denoted $\|\cdot\|_{X_h}$ and $\|\cdot\|_{Q_h}$; on $X_h\oplus Q_h$ we use $\|(u,q)\|_h^2=\|u\|_{X_h}^2+\|q\|_{Q_h}^2$. The induced operator norm is $\|\cdot\|_{h\to h}$, without a $\delta$-dependent weight. With periodic indices, a concrete pair is
\begin{equation}
 (G_hu)_{i+1/2}=(u_{i+1}-u_i)/h,\qquad
 (D_hq)_i=(q_{i+1/2}-q_{i-1/2})/h.
\label{eq:staggered-example}
\end{equation}
Summation by parts gives $\langle D_hq,u\rangle_{X_h}=-\langle q,G_hu\rangle_{Q_h}$, i.e. $D_h=-G_h^*$. This is the periodic weighted adjoint structure related to summation-by-parts discretizations \cite{Strand1994}; we do not claim the complete boundary closure of a nonperiodic SBP scheme. Higher-order paired operators used below obey the same adjoint identity. Section~\ref{sec:fv} proves it for the main central-flux finite-volume code path; the Rusanov branch does not have the required zero scalar block.
\end{revision}
 The one-dimensional semi-discrete matrix is
\begin{equation}
 A_{h,\delta}^{1D}=
 \begin{bmatrix}
 0&-D_h\\
 -(a/\delta)G_h&-(1/\delta)I
 \end{bmatrix},
\label{eq:A1d}
\end{equation}
The lower-left transport block and local relaxation block scale as $\delta^{-1}$, while the upper-right divergence block does not. This is the meaning of fully stiff here. The paired two-dimensional matrix for the same equation is constructed in Section~\ref{sec:two-dimensional-main}.

\begin{lemma}[Semi-discrete energy dissipation]
\label{lem:energy}
Assume that the discrete operators satisfy $D_h=-G_h^*$ in the underlying quadrature inner products. Then
\begin{equation}
 \frac{d}{dt}\,\mathcal E_h(t)=-\|\qq_h(t)\|^2,
 \qquad
 \mathcal E_h(t)=\frac12\left(a\|u_h(t)\|^2+\delta\|\qq_h(t)\|^2\right).
\label{eq:energy-law}
\end{equation}
\end{lemma}
\begin{proof}
Take the scalar product of $u_t=-D_h\qq$ with $a u$ and of $\delta\qq_t=-aG_hu-\qq$ with $\qq$.  Since $D_h=-G_h^*$, the transport terms cancel and the remaining term is $-\|\qq\|^2$.
\end{proof}
Thus the compatible spatial pair reproduces the physical relaxation dissipation before time discretization.  L-stability is then used to damp the unresolved discrete fast spectrum without corrupting the slow diffusive component.

\begin{revision}For a nonzero singular value $\kappa_h$ of $G_h$, the scalar and the flux component parallel to its singular vector form the $2\times2$ matrix $\left[\begin{smallmatrix}0&\kappa_h\\-a\kappa_h/\delta&-1/\delta\end{smallmatrix}\right]$. Its two eigenvalues, called slow and fast, are\end{revision}
\begin{equation}
 \lambda_{s,f}(\delta)=\frac{-1\pm\sqrt{1-4a\delta\kappa_h^2}}{2\delta}.
\end{equation}
Thus
\begin{equation}
 \lambda_s=-a\kappa_h^2+O(\delta),\qquad
 \lambda_f=-\delta^{-1}+a\kappa_h^2+O(\delta).
\label{eq:slow-fast}
\end{equation}
In two dimensions, a transverse flux mode has eigenvalue $-1/\delta$. The slow mode must remain accurate; the fast modes should disappear when they are not resolved.

\section{Implicit two-stage fourth-order temporal discretization}
\subsection{Two stages and the trajectory derivative}
For the autonomous evolution equation
\begin{equation}
 \UU_t=\LL(\UU),
\end{equation}
define the first trajectory derivative
\begin{equation}
 \GG(\UU)=\frac{\dd}{\dd t}\LL(\UU(t))=D\LL(\UU)\,\LL(\UU).
\label{eq:G}
\end{equation}
Let $\UU^\star$ approximate $\UU(t_n+\Dt/2)$. The first stage is
\begin{equation}
\UU^\star=\UU^n+\frac{\Dt}{4}\bigl[\LL(\UU^n)+\LL(\UU^\star)\bigr]
+\frac{\Dt^2}{48}\bigl[\GG(\UU^n)-\GG(\UU^\star)\bigr].
\label{eq:stage1}
\end{equation}
The complete step is
\begin{align}
\UU^{n+1}={}&\UU^n+\Dt\bigl[a_0\LL(\UU^n)+a_1\LL(\UU^\star)+a_2\LL(\UU^{n+1})\bigr]\nonumber\\
&+C\Dt^2\bigl[\GG(\UU^n)-\GG(\UU^\star)-\tfrac32\GG(\UU^{n+1})\bigr],
\label{eq:stage2}
\end{align}
with
\begin{equation}
 a_0=\frac16+\frac72C,\qquad
 a_1=\frac23-10C,\qquad
 a_2=\frac16+\frac{13}{2}C.
\label{eq:coeffs}
\end{equation}
\begin{revision}Equations \eqref{eq:stage1}--\eqref{eq:coeffs} and the following admissible interval are taken from \cite{Huo2026}; \begin{newrevision}their conservative finite-volume realization and completed-step relaxation limit are studied below.\end{newrevision} The L-stable interval derived there is\end{revision}
\begin{equation}
 C\in[C_-,C_+],\qquad
 C_{\pm}=\frac{25\pm\sqrt{105}}{780}.
\label{eq:Cinterval}
\end{equation}

\subsection{Temporal facts recalled from the earlier method}
\begin{revision}
For the exact trajectory derivative \eqref{eq:G}, the two-stage formula is fourth order for sufficiently smooth autonomous systems, and it is A- and L-stable for the interval \eqref{eq:Cinterval}; the order conditions and scalar stability proof are in \cite{Huo2026}. We do not repeat those proofs. The half-stage and complete-step scalar factors used in the PDE analysis are
\end{revision}
\begin{equation}
 H(z)=\frac{1+z/4+z^2/48}{1-z/4+z^2/48}.
\label{eq:H}
\end{equation}
\begin{equation}
 R(z)=\frac{1+a_0z+Cz^2+(a_1z-Cz^2)H(z)}{1-a_2z+\tfrac32Cz^2}.
\label{eq:R}
\end{equation}
For the parameter comparison below, the equivalent reduced form is
\begin{equation}
 R(z)=-\frac{(183C-5)z^3+(972C-42)z^2+(1872C-168)z-288}
 {(z^2-12z+48)\,[9Cz^2-(39C+1)z+6]}.
\label{eq:Rreduced}
\end{equation}
\begin{revision}
The left-half-plane bound $|R(z)|\le1$, the absence of poles there, and $R(z)\to0$ at infinity are used in later proofs. The scalar stability contour is retained in the supplementary material as a visual diagnostic, not a separate PDE result.
\end{revision}
\subsection{A parameter value with quadratic stiff decay}
The free parameter controls both the local error constant and the asymptotic damping rate. \begin{newrevision}The following expansion characterizes this tradeoff for the inherited stability function \eqref{eq:Rreduced}; the cancellation at $C_q$ and its fixed-parameter scope define the damping comparison in this paper.\end{newrevision}
\begin{proposition}[Accuracy--damping expansion]\label{prop:accuracy-damping}
As $z\to0$,
\begin{equation}
 R(z)-e^z=-\frac{75C-1}{2880}z^5
 -\frac{5850C^2+525C-7}{34560}z^6+O(z^7).
\label{eq:small-expansion}
\end{equation}
As $|z|\to\infty$,
\begin{equation}
 R(z)=\left(-\frac{61}{3}+\frac{5}{9C}\right)\frac1z+O(z^{-2}).
\label{eq:large-expansion}
\end{equation}
Consequently,
\begin{equation}
 C_q=\frac5{183}
\label{eq:Cq}
\end{equation}
cancels the $1/z$ term and gives
\begin{equation}
 R(z)=\frac{314}{5z^2}+O(z^{-3}).
\label{eq:quadratic-decay}
\end{equation}
\end{proposition}
\begin{proof}
The Taylor expansion at the origin is obtained by writing the denominator in \eqref{eq:Rreduced} as a power series with nonzero constant term and matching coefficients through degree six; subtraction of the exponential series gives \eqref{eq:small-expansion}.

For the stiff expansion, the denominator polynomial is
\begin{equation}
 D_C(z)=9Cz^4-(147C+1)z^3+(900C+18)z^2-(1872C+120)z+288.
\end{equation}
The leading numerator term is $-(183C-5)z^3$.  Polynomial division therefore gives
\begin{equation}
 [z^{-1}]R=-\frac{183C-5}{9C}=-\frac{61}{3}+\frac{5}{9C},
\end{equation}
which proves \eqref{eq:large-expansion}.  This coefficient vanishes only at $C=5/183$.  At that value the cubic numerator term is absent, and the next numerator coefficient together with the leading denominator coefficient gives
\begin{equation}
 \lim_{z\to\infty}z^2R(z)
 =-\frac{972(5/183)-42}{9(5/183)}=\frac{314}{5}.
\end{equation}
Hence \eqref{eq:quadratic-decay} follows.
\end{proof}

\begin{newrevision}In this subsection, \textit{optimal} refers only to the asymptotic damping exponent as $x\to\infty$ within the fixed admissible parameter family; it does not mean optimal finite-stiffness accuracy, overall accuracy, or computational efficiency.\end{newrevision}
\begin{revision}
\begin{corollary}[Unique optimal deep-stiff decay order within the fixed-parameter family]
\label{cor:unique-quadratic-damping}
For a fixed $C\in[C_-,C_+]$, put $x=-z>0$. If $C\ne C_q$, then
$\lim_{x\to\infty}x|R(-x;C)|=|-61/3+5/(9C)|>0$. At $C=C_q$,
$\lim_{x\to\infty}x^2|R(-x;C_q)|=314/5>0$.
Thus $C_q$ is the unique admissible \emph{fixed} parameter with quadratic
inverse-stiffness decay; no member has cubic or higher inverse-stiffness
decay. This optimality concerns the asymptotic exponent, not uniform
fourth-order accuracy, finite-stiffness amplitude, or computational cost.
\end{corollary}
\begin{proof}
The coefficient of $1/z$ in \eqref{eq:large-expansion} is zero only at
$C_q$, which lies strictly inside \eqref{eq:Cinterval}. The nonzero
$1/z^2$ coefficient in \eqref{eq:quadratic-decay} precludes any higher
decay exponent for that sole candidate.
\end{proof}

For a \emph{single finite} stiffness $x$, the numerator of
\eqref{eq:Rreduced} can instead be zero at
\begin{equation}
 C_x=\frac{5x^3-42x^2+168x-288}{183x^3-972x^2+1872x},
 \qquad C_x\longrightarrow C_q\quad(x\to\infty).
 \label{eq:finite-x-zero}
\end{equation}
When admissible, $C_x$ cancels the amplification of that \emph{one}
scalar mode at that \emph{one} stiffness. It varies with $x$ and is not
an alternative fixed coefficient with the uniform deep-stiff exponent
of Corollary~\ref{cor:unique-quadratic-damping}. The numerical scan below
therefore compares fixed coefficients over a range of stiffness values
instead of selecting the minimum at a single $x$.
\end{revision}

All three parameters lie in the L-stable interval, but their roles differ. The endpoint $C_-$ has the smallest fifth-order error coefficient among the three values used here; $C_q$ has the uniquely strongest asymptotic decay exponent; and $C_+$ is the upper endpoint of the proven interval. Section~\ref{sec:quadratic-damping-fv} tests this exponent with the actual Cartesian finite-volume transverse mode over a stiffness sweep, then measures the full-state error of a mixed slow--fast field against the exact positive-$\delta$ semi-discrete trajectory; auxiliary parameter diagnostics appear in the supplement.
\begin{newrevision}
The earlier temporal analysis \cite{Huo2026} recommended $C_+$ using an
unnormalized imaginary-axis stability margin. That criterion differs from
the asymptotic decay of $|R(-x;C)|$ for negative-real relaxation eigenvalues.
The comparison here keeps $C$ fixed as $x\to\infty$:
Corollary~\ref{cor:unique-quadratic-damping} singles out $C_q$ for its
quadratic decay exponent, and the two-dimensional FV sweep tests its
effect on transverse and mixed fields. The half-stage factor remains
$H(z)$ as defined in \eqref{eq:H}; the symbol $K$ is reserved for the
stage-coupling matrix in Section~6.2.
\end{newrevision}
\subsection{Practical selection of the free parameter}
The parameter should be selected according to the dominant error mechanism rather than by a single universal rule. For smooth moderately stiff dynamics, $C_-$ is the natural default because its fifth-order defect is the smallest among the three representative values. When the solution contains a strongly damped initial layer or transverse relaxation mode, $C_q$ can be preferable because the asymptotic amplification decays with one additional inverse power, $O(|z|^{-2})$ instead of $O(|z|^{-1})$. The two-dimensional mixed-field test below illustrates a lower one-step full-state error for that particular slow--fast data; it does not establish an error optimum for all initial states. The endpoint $C_+$ remains useful as a member of the previously analyzed interval and is retained in some damping comparisons because it was emphasized in the original stability study.

The same parameter should normally be used in both stages and throughout a calculation. Switching $C$ according to a local stiffness indicator would change the method into a variable-coefficient integrator and would require a separate order and stability analysis. A safer adaptive strategy changes $\Dt$ while keeping $C$ fixed. In the numerical sections, $C_-$ is used for accuracy and nonlinear convergence tests, $C_+$ retains a direct comparison with classical stiff decay, and all three fixed values are compared over the same transverse-mode stiffness sweep. In a mixed slow--fast solution, $C_q$ should be chosen when damping of deeply unresolved fast components is the objective; the $z^5$ coefficient in \eqref{eq:small-expansion} still governs its distinct smooth-mode accuracy cost.

\section{Conservative finite-volume operators and discrete trajectory closure}\label{sec:fv}
\subsection{One-dimensional finite-volume residual}
For a control volume $I_i=[x_{i-1/2},x_{i+1/2}]$, define
\begin{equation}
 \bar\UU_i(t)=\frac1{\Delta x_i}\int_{I_i}\UU(x,t)\,\dd x.
\end{equation}
A conservative semi-discrete residual is \cite{ToroBook}
\begin{equation}
 \LL_{h,i}(\bm Z)=-\frac{\widehat\FF_{i+1/2}(\bm Z)-\widehat\FF_{i-1/2}(\bm Z)}{\Delta x_i}
 +\bar\SSS_i(\bm Z).
\label{eq:fv-L}
\end{equation}
The face-based trajectory-derivative approximation supplied to the temporal method is assembled as
\begin{equation}
 \widetilde{\GG}_{h,i}(\bm Z)=-\frac{\widehat\FF_{t,i+1/2}(\bm Z)-\widehat\FF_{t,i-1/2}(\bm Z)}{\Delta x_i}
 +\bar\SSS_{t,i}(\bm Z).
\label{eq:fv-G}
\end{equation}
Identical face values are used by adjacent cells, so both $\LL_h$ and $\widetilde{\GG}_h$ are conservative flux differences.

Applying the first sequential time stage \eqref{eq:stage1} to the semi-discrete pair \eqref{eq:fv-L}--\eqref{eq:fv-G} gives
\begin{align}
\bar\UU_i^\star={}&\bar\UU_i^n
-\frac{\Dt}{4\Delta x_i}\Big[(\widehat\FF_{i+1/2}^n-\widehat\FF_{i-1/2}^n)
 +(\widehat\FF_{i+1/2}^\star-\widehat\FF_{i-1/2}^\star)\Big]\nonumber\\
&+\frac{\Dt}{4}(\bar\SSS_i^n+\bar\SSS_i^\star)
-\frac{\Dt^2}{48\Delta x_i}\Big[(\widehat\FF_{t,i+1/2}^n-\widehat\FF_{t,i-1/2}^n)
 -(\widehat\FF_{t,i+1/2}^\star-\widehat\FF_{t,i-1/2}^\star)\Big]\nonumber\\
&+\frac{\Dt^2}{48}(\bar\SSS_{t,i}^n-\bar\SSS_{t,i}^\star).
\label{eq:fv-stage1}
\end{align}
The second stage is obtained analogously with the weights in \eqref{eq:stage2}. \begin{revision}The unknown-stage face fluxes couple neighboring cells, so both stages require global solves. For the fixed linear reconstruction and flux studied in the main article these are \emph{linear} systems; a nonlinear flux would produce global nonlinear systems.\end{revision}

\subsection{MOL-consistent finite-volume trajectory-derivative provider}
For the semi-discrete ODE $Z_t=\mathcal L_h(Z)$, the exact trajectory derivative required by the fourth-order temporal formula is
\begin{equation}
 \mathcal G_h^{\rm tr}(Z):=D\mathcal L_h(Z)\,\mathcal L_h(Z).
\label{eq:exact-discrete-G}
\end{equation}
\begin{revision}
The supplied implementation uses the following exact derivative construction for its specified linear finite-volume residual. For the one-dimensional constant-coefficient form $\UU_t+A\UU_x=S\UU$, let $Z_i$ be the two-component periodic cell average in cell $i$ and let $N$ be the number of cells. Thus $Z=(Z_0,\ldots,Z_{N-1})$ has $N_{\rm dof}=2N$ scalar components. At face $i+1/2$ reconstruct the left and right states with
\begin{align}
 Z^-_{i+1/2}&=(2Z_{i-2}-13Z_{i-1}+47Z_i+27Z_{i+1}-3Z_{i+2})/60,\nonumber\\
 Z^+_{i+1/2}&=(-3Z_{i-1}+27Z_i+47Z_{i+1}-13Z_{i+2}+2Z_{i+3})/60.
\label{eq:fv-reconstruction}
\end{align}
The fixed linear numerical flux used by the implementation is
\begin{equation}
 \widehat\FF_{i+1/2}(Z)=\tfrac12A(Z^-_{i+1/2}+Z^+_{i+1/2})
 -\tfrac12\alpha(Z^+_{i+1/2}-Z^-_{i+1/2}),
\label{eq:rusanov-specified}
\end{equation}
with $\alpha=0$ for the main central branch and $\alpha=\sqrt{a/\delta}$ for the Rusanov comparison; $\bar\SSS_i=SZ_i$. Once the face residual \eqref{eq:fv-L} has been formed, compute the \emph{cell} trajectory velocity $V_i=\LL_{h,i}(Z)$, reconstruct $V^\pm$ by the same weights \eqref{eq:fv-reconstruction}, and set
\begin{equation}
 \widehat\FF_{t,i+1/2}=\tfrac12A(V^-_{i+1/2}+V^+_{i+1/2})
 -\tfrac12\alpha(V^+_{i+1/2}-V^-_{i+1/2}),\qquad
 \bar\SSS_{t,i}=SV_i.
\label{eq:flux-t}
\end{equation}
The shared face value in \eqref{eq:fv-G} is the chain-rule derivative of the \emph{same discrete flux along the same semi-discrete trajectory}. Periodic indexing closes all stencils. The code's map \path{LinearBalanceFV1D.trajectory_derivative} implements precisely these operations for both choices of $\alpha$, and the numerical stage solver receives the separately assembled $\widetilde{\GG}_h$.

For $\alpha=0$, averaging the two reconstructed traces at a face gives $(Z_{i-2}-8Z_{i-1}+37Z_i+37Z_{i+1}-8Z_{i+2}+Z_{i+3})/60$. Its face difference defines the periodic cell derivative
\begin{equation}
 (T_hz)_i=\frac{-z_{i-3}+9z_{i-2}-45z_{i-1}+45z_{i+1}-9z_{i+2}+z_{i+3}}{60h},
 \qquad T_h^*=-T_h.
\label{eq:central-fv-pair}
\end{equation}
Thus $\mathcal L_h(u,q)=(-T_hq,-(aT_hu+q)/\delta)$ is \emph{exactly} \eqref{eq:A1d} with $G_h=D_h=T_h$ and $D_h=-G_h^*$, including the absence of an upper-left scalar block. The $h$-weighted cell inner product is used for both fields. In particular, Theorem~\ref{thm:fullstep-ap} applies to the actual central-flux finite-volume stage matrices and to its face-assembled trajectory derivative. Like other central fluxes this linear smooth-solution branch is not claimed to be monotone for discontinuities.

The Rusanov choice still satisfies the closure identity below but introduces an upper-left numerical-viscosity block $B_{h,\delta}$. At a fixed grid its coefficient is proportional to $\sqrt{a/\delta}$, and $\|B_{h,\delta}\|$ grows as $\delta^{-1/2}$ whenever the reconstruction jump is nonzero. It therefore cannot be put into \eqref{eq:A1d} with a $\delta$-independent compatible pair, and the AP theorem makes no claim for this branch. Section~\ref{sec:fv-ap-data} computes both complete one-step matrices against the same discrete heat target.
\begin{proposition}[\textcolor{red!75!black}{Wave-speed Rusanov flux is not fixed-grid AP for the specified central-flux discrete heat target}]\label{prop:rusanov-not-ap}
On an even periodic grid, the fifth-order Rusanov finite-volume residual with $\alpha=\sqrt{a/\delta}$ does not satisfy the full-step AP definition in Section~5 for the target $E_hR(\Delta t aT_h^2)P_h$. \begin{newrevision}The proposition is restricted to this specified central-flux discrete heat target and fixed-grid operator; it is not a general impossibility result for Rusanov-type discretizations.\end{newrevision}
\end{proposition}
\begin{proof}
Take the scalar checkerboard $u_i=(-1)^i$ and $q_i=0$. The central derivative $T_hu$ vanishes, so the target heat step leaves this scalar unchanged. At each face the left and right reconstructions of $u$ are $(8/15)u_i$ and $-(8/15)u_i$, respectively. Their central average is zero, whereas the Rusanov jump term gives the scalar residual $u_t=-16\alpha u/(15h)$ on this mode. The coupled central transport terms vanish on the same mode. Its full-step scalar factor is therefore $R(-16\alpha\Delta t/(15h))$, which tends to zero by L-stability as $\delta\to0$. The discrepancy from the target is one on a normalized checkerboard vector, so the operator-norm error cannot tend to zero.
The checkerboard is a grid-scale mode in the kernel of this central difference. Its use is legitimate for the stated \emph{operator-norm, fixed-grid} AP definition, which covers every discrete mode; it does not claim that the central discretization accurately resolves this nonsmooth mode, nor rule out a different well-designed AP flux and limiting heat discretization.
\end{proof}

For nonlinear fluxes or state-dependent reconstruction, the same choice would require the full Jacobian action $D\LL_h(Z)\LL_h(Z)$, including derivatives of reconstruction weights and numerical fluxes. The present identity uses fixed linear maps; the ancillary nonlinear example in the supplement computes the Jacobian action analytically for its own discrete residual. This is consistent with the earlier discrete-derivative remedy of \cite{SchutzSealJaust2017}.
\end{revision}

\begin{algorithm}[H]
\caption{\textcolor{black}{Exact semi-discrete trajectory derivative at one stage state $Z$ for the specified linear face flux}}
\label{alg:trajectory-provider}
\begin{algorithmic}[1]
\State Choose $\alpha=0$ for the AP-compatible central branch or $\alpha=\sqrt{a/\delta}$ for the Rusanov comparison; reconstruct $Z^\pm$ by \eqref{eq:fv-reconstruction} and assemble $V_i=\LL_{h,i}(Z)$ using \eqref{eq:rusanov-specified} and $SZ_i$.
\State Reconstruct $V^\pm$ with the same periodic linear stencil.
\State Apply \eqref{eq:flux-t} at every shared face and the cell source $SV_i$.
\State Difference the face derivatives as in \eqref{eq:fv-G} and return $(\LL_h(Z),\widetilde{\GG}_h(Z))$ to the current implicit stage.
\end{algorithmic}
\end{algorithm}

For the linear verification problem $\UU_t+A\UU_x=S\UU$, let $\bm V_h=\LL_h(\UU_h)$.  Linearity of the reconstruction and numerical flux gives
\begin{equation}
 \widehat\FF_t=\widehat\FF(\bm V_L,\bm V_R),\qquad \SSS_t=S\bm V,
\label{eq:linear-trajectory-flux}
\end{equation}
which is the exact constant-coefficient face derivative along the selected semi-discrete trajectory. It yields the exact identity
\begin{equation}
 \widetilde{\GG}_h(\UU_h)=\GG_h^{\rm tr}(\UU_h)=\LL_h(\bm V_h)=\LL_h^2(\UU_h).
\label{eq:A2closure}
\end{equation}

\begin{proposition}[Linear face-derivative closure]\label{prop:fv-closure}
Consider $\UU_t+A\UU_x=S\UU$ with constant matrices, a linear reconstruction, and a linear numerical flux.  If $\widehat\FF_t$ and $\SSS_t$ are defined by \eqref{eq:linear-trajectory-flux}, then the conservative assembly \eqref{eq:fv-G} satisfies \eqref{eq:A2closure} exactly.
\end{proposition}
\begin{proof}
\begin{newrevision}
Write the residual as $\LL_h=\mathcal C_h\mathcal F_hR_h+S_h$, where $R_h$ reconstructs cell averages to face states, $\mathcal F_h$ applies the linear numerical flux at each face, and $\mathcal C_h$ takes conservative face differences. The face-difference map $\mathcal C_h$ is distinct from the scalar projection $P_h$ used in the AP analysis. All three maps are linear. Along the semi-discrete trajectory $\bm V_h=\LL_h(\UU_h)$,
\[
 \frac{\dd}{\dd t}\LL_h(\UU_h)
 =\mathcal C_h\mathcal F_hR_h\bm V_h+S_h\bm V_h
 =\LL_h(\bm V_h)=\LL_h^2(\UU_h).
\]
The left-hand side is precisely the assembly of the face flux derivatives and source derivatives in \eqref{eq:fv-G}.
\end{newrevision}
\end{proof}

\begin{revision}
The reference implementation keeps the two matrices separate: $L_h$ is assembled from \eqref{eq:fv-reconstruction}--\eqref{eq:rusanov-specified}, while $\widetilde G_h$ repeats exactly the same reconstruction and flux operation on $V_h=L_hU_h$. The stage equations use $\widetilde G_h$ directly; $L_h^2$ is independently formed only for an audit. This is the precise MOL-consistent instantaneous construction used here. The central branch additionally satisfies \eqref{eq:central-fv-pair}; exact derivative closure \emph{alone} does not make the Rusanov branch AP.
\end{revision}

\subsection{Required spatial accuracy and a global perturbation estimate}
Let $\Pi_h$ be the projection to the spatial discrete space.  For a smooth exact solution, assume residual and exact-trajectory consistency,
\begin{align}
\|\LL_h(\Pi_h\UU)-\Pi_h\LL(\UU)\|&\le C h^p,\\
\|\GG_h^{\rm tr}(\Pi_h\UU)-\Pi_h(D\LL(\UU)\LL(\UU))\|&\le C h^p,
\label{eq:spatial-consistency}
\end{align}
and assume that any supplied approximate derivative provider, when used, satisfies
\begin{equation}
 \|\widetilde{\GG}_h(\Pi_h\UU)-\GG_h^{\rm tr}(\Pi_h\UU)\|\le C h^p.
\label{eq:trajectory-closure-consistency}
\end{equation}
Let $\Phi_{h,\Delta t}$ denote one exact numerical step formed with $\GG_h^{\rm tr}$, and let $\widetilde\Phi_{h,\Delta t}$ denote the step actually returned after replacing it by $\widetilde{\GG}_h$, approximating residual data, and terminating the nonlinear iterations.

\begin{theorem}[Global space--time, trajectory-closure, and solver error]
\begin{revision}
At fixed positive $\delta$, assume the solution has bounded time derivatives through order five, the semi-discrete residual is sufficiently smooth near its projection, the exact stage maps are locally unique, and $\UU_h^0=\Pi_h\UU(0)$. For the joint $h$--$\Delta t$ bound below, take the constants in \eqref{eq:spatial-consistency}, the local stage inverse bounds, and the following step Lipschitz estimate uniformly over the grids considered at this fixed $\delta$; without these uniform assumptions the assertion is only a fixed-grid temporal estimate. Assume, in a neighborhood of the projected exact solution,
\[
 \|\Phi_{h,\Delta t}(V)-\Phi_{h,\Delta t}(W)\|\le(1+L\Delta t)\|V-W\|.
\]
Suppose the one-step perturbation satisfies
\[
 \|\widetilde\Phi_{h,\Delta t}(V)-\Phi_{h,\Delta t}(V)\|
 \le C\bigl(\Delta t\,\eta_L+\Delta t^2\eta_G+\eta_{\rm solve}\bigr).
\]
Then, for $t_n\le T$ and sufficiently small $h$ and $\Delta t$,
\begin{equation}
 \|\Pi_h\UU(t_n)-\UU_h^n\|
 \le C_T\left(h^p+\Delta t^4+\eta_L+\Delta t\eta_G+\frac{\eta_{\rm solve}}{\Delta t}\right).
\label{eq:total-error}
\end{equation}
Here $\eta_G$ includes the difference between the supplied derivative and $\GG_h^{\rm tr}$. The left-hand norm compares grid data with the \emph{projected} PDE solution, as in the experiments; a continuous-space estimate would require a separately specified reconstruction and projection error. Exact fourth-order temporal convergence to the semi-discrete ODE at fixed $h$ follows when $\eta_L=O(\Delta t^4)$, $\eta_G=O(\Delta t^3)$, and $\eta_{\rm solve}=O(\Delta t^5)$. For a smooth PDE calculation, the closure condition \eqref{eq:trajectory-closure-consistency} gives $\eta_G=O(h^p)$, which is absorbed into the combined $O(h^p+\Delta t^4)$ projected space--time error under the stated mesh-uniform assumptions at fixed $\delta$. No constant is asserted to be uniform as $\delta\to0$.
\end{revision}
\end{theorem}
\begin{proof}
\begin{revision}
Set $v(t)=\Pi_h\UU(t)$. The two consistency inequalities in \eqref{eq:spatial-consistency} give
\begin{equation}
 \|v'(t)-\LL_h(v(t))\|\le Ch^p,
 \qquad \|v''(t)-\GG_h^{\rm tr}(v(t))\|\le Ch^p.
\end{equation}
Insert $v(t_n)$, $v(t_n+\Dt/2)$, and $v(t_{n+1})$ into the two stage equations. Fourth-order temporal consistency at fixed $h,\delta$ gives an $O(\Dt^5)$ defect in each exact-trajectory stage equation; the displayed mismatches add $O(\Dt h^p+\Dt^2h^p)$. For sufficiently small $\Dt$ at fixed $h,\delta$, the linearizations of the sequential stage equations are identity plus $O(\Dt)$ and have bounded local inverses. The stage defects therefore imply the projected one-step defect
\[
 \|\Pi_h\UU(t_{n+1})-\Phi_{h,\Delta t}(\Pi_h\UU(t_n))\|
 \le C(\Delta t^5+\Delta t h^p).
\]
Add and subtract the exact numerical step at $\UU_h^n$, use the Lipschitz bound and the assumed one-step implementation perturbation, and obtain
\[
 e_{n+1}\le(1+L\Delta t)e_n+C(\Delta t^5+\Delta t h^p+\Delta t\eta_L+\Delta t^2\eta_G+\eta_{\rm solve}).
\]
Discrete Gronwall over at most $T/\Delta t$ steps proves \eqref{eq:total-error}.
\end{revision}
\end{proof}

The theorem explains the requirement placed on the spatial module in the temporal-order tests: its consistency error must remain below the $O(\Delta t^4)$ contribution.  For state-dependent WENO or Hermite reconstructions, a corresponding derivative must include the dependence of their reconstruction weights and interface data on the stage state.  At physical boundaries we use stage-consistent ghost or face data following \cite{DuLiBoundary2018}; no boundary construction is developed here.

\section{Rigorous full-step asymptotic-preserving analysis}
\subsection{Definition and compatible space operators}
AP and L-stability concern different limits. L-stability holds for a fixed semi-discrete eigenvalue as $z=\lambda\Dt$ becomes large. AP concerns the singular physical parameter $\delta\to0$ at fixed $h$ and $\Dt$.

\begin{revision}
The spaces, weighted inner products, product norm, and concrete periodic pair were defined in Section~2. The following theorem also applies to any fixed finite-dimensional $X_h,Q_h$ with those induced norms and an adjoint pair. In equations below $\|\cdot\|$ means the product norm $\|\cdot\|_h$ or its induced operator norm $\|\cdot\|_{h\to h}$, as appropriate. Let $G_h:X_h\to Q_h$ and $D_h:Q_h\to X_h$ satisfy
\end{revision}
\begin{equation}
 D_h=-G_h^*.
\label{eq:compatible}
\end{equation}
Define
\begin{equation}
 A_{h,\delta}=
 \begin{bmatrix}
 0&-D_h\\
 -(a/\delta)G_h&-(1/\delta)I
 \end{bmatrix},\qquad
 L_{D,h}=-aG_h^*G_h,
\label{eq:APoperators}
\end{equation}
and the projection and equilibrium embedding
\begin{equation}
 P_h(u,q)^T=u,\qquad E_hu=(u,-aG_hu)^T.
\label{eq:PE}
\end{equation}
The limiting diffusion operator is negative semidefinite because
\begin{equation}
 \langle u,L_{D,h}u\rangle=-a\|G_hu\|^2\le0.
\end{equation}

\begin{definition}
The time-discrete one-step map applied to the semi-discrete system is \emph{full-step AP} if, for fixed $h$ and $\Dt$,
\begin{equation}
 R(\Dt A_{h,\delta})\longrightarrow E_hR(\Dt L_{D,h})P_h
\end{equation}
as $\delta\to0$ in operator norm.
\end{definition}

\subsection{Orthogonal slow--fast decomposition}
\begin{lemma}[Block decomposition]\label{lem:block-decomposition}
Let $G_h=U\Sigma V^*$ be a singular-value decomposition with positive singular values $\sigma_1,\dots,\sigma_r$. Under the unitary basis transformation $\operatorname{diag}(V,U)$, $A_{h,\delta}$ is the orthogonal direct sum of zero blocks on $\ker G_h$, scalar blocks $-1/\delta$ on $\ker G_h^*$, and
\begin{equation}
 B_{j,\delta}=
 \begin{bmatrix}
 0&\sigma_j\\
 -a\sigma_j/\delta&-1/\delta
 \end{bmatrix},\qquad j=1,\dots,r.
\label{eq:Bblock}
\end{equation}
\end{lemma}
\begin{proof}
Write $u=V\widehat u+u_0$ with $u_0\in\ker G_h$ and $q=U\widehat q+q_0$ with $q_0\in\ker G_h^*$. The identities $G_hV=U\Sigma$ and $-D_h=G_h^*=V\Sigma^*U^*$ couple only the coefficients $(\widehat u_j,\widehat q_j)$. The kernel components satisfy $u_{0,t}=0$ and $q_{0,t}=-q_0/\delta$.
\end{proof}

\begin{lemma}[Eigenvalues and projectors]
For sufficiently small $\delta$, the block \eqref{eq:Bblock} has distinct negative eigenvalues
\begin{equation}
 \lambda_{s,j}=\frac{-1+\sqrt{1-4a\delta\sigma_j^2}}{2\delta},\qquad
 \lambda_{f,j}=\frac{-1-\sqrt{1-4a\delta\sigma_j^2}}{2\delta}.
\end{equation}
They satisfy
\begin{equation}
 \lambda_{s,j}=-a\sigma_j^2+O(\delta),\qquad
 \lambda_{f,j}=-\delta^{-1}+a\sigma_j^2+O(\delta).
\label{eq:eigexpand}
\end{equation}
The slow and fast spectral projectors are uniformly bounded for sufficiently small $\delta$, and
\begin{equation}
 \Pi_{s,j}=\frac{B_{j,\delta}-\lambda_{f,j}I}{\lambda_{s,j}-\lambda_{f,j}}
 =\begin{bmatrix}1&0\\-a\sigma_j&0\end{bmatrix}+O(\delta).
\label{eq:projector}
\end{equation}
\end{lemma}
\begin{proof}
The characteristic polynomial is $\lambda^2+\delta^{-1}\lambda+a\sigma_j^2/\delta$. Taylor expansion of the square root gives \eqref{eq:eigexpand}. Moreover,
\begin{equation}
 \lambda_{s,j}-\lambda_{f,j}=\delta^{-1}\sqrt{1-4a\delta\sigma_j^2},
\end{equation}
which is bounded below by $c/\delta$ for the finite set of singular values and sufficiently small $\delta$. Substitution into the explicit projector formula proves \eqref{eq:projector} and uniform boundedness. Since $\Pi_{f,j}=I-\Pi_{s,j}$, the fast projector is also uniformly bounded.
\end{proof}

\subsection{Stability estimates and the one-step theorem}
\begin{lemma}[Bounds for the rational time map]\label{lem:rational-bounds}
For every L-stable parameter in \eqref{eq:Cinterval}, there exist $M,z_0>0$ such that
\begin{equation}
 |R(z)|\le \frac{M}{|z|},\qquad z\le-z_0.
\label{eq:Rlarge}
\end{equation}
For $C=C_q$, the stronger estimate $|R(z)|\le M/|z|^2$ holds. On every compact subset of the negative real axis, $R$ is Lipschitz continuous.
\end{lemma}
\begin{proof}
The estimates follow from \eqref{eq:large-expansion} and \eqref{eq:quadratic-decay}. Rational differentiability and absence of poles on the negative real axis give the compact Lipschitz estimate.
\end{proof}

\begin{lemma}[Uniform boundedness of the full one-step map]\label{lem:uniform-step}
For fixed $h$ and $\Dt$, there are $\delta_0>0$ and $M_{h,\Dt}$ such that
\begin{equation}
 \sup_{0<\delta\le\delta_0}\|R(\Dt A_{h,\delta})\|\le M_{h,\Dt}.
\label{eq:uniform-step}
\end{equation}
\end{lemma}
\begin{proof}
Use the orthogonal direct-sum representation. On $\ker G_h$, the map is the identity. On $\ker G_h^*$, it is multiplication by $R(-\Dt/\delta)$, which is bounded by A-stability. On each coupled block,
\begin{equation}
 R(\Dt B_{j,\delta})=R(\Dt\lambda_{s,j})\Pi_{s,j}+R(\Dt\lambda_{f,j})\Pi_{f,j}.
\end{equation}
The projectors are uniformly bounded, the slow argument stays in a fixed compact subset of the negative axis, and the fast factor is bounded by A-stability. The finite maximum over all blocks proves \eqref{eq:uniform-step}.
\end{proof}

\begin{theorem}[Full-step AP operator estimate]\label{thm:fullstep-ap}
\begin{newrevision}For the linear constant-coefficient relaxation system on a fixed finite-dimensional periodic grid, \end{newrevision}assume \eqref{eq:compatible}, exact closure $\GG_h(\UU)=A_{h,\delta}^2\UU$, exact solution of the two implicit stages, and an L-stable parameter $C$.  For fixed finite-dimensional spaces and fixed $\Dt>0$, there are $\delta_0>0$ and $K_{h,\Dt}>0$ such that
\begin{equation}
 \left\|R(\Dt A_{h,\delta})-E_hR(\Dt L_{D,h})P_h\right\|
 \le K_{h,\Dt}\delta,
\label{eq:AP-estimate}
\end{equation}
for $0<\delta\le\delta_0$. \begin{newrevision}The estimate is a fixed-grid, periodic, linear full-step result with exact stage solutions; its constant may depend on the grid, and it makes no mesh-uniform or nonlinear AP claim.\end{newrevision}
\end{theorem}
\begin{proof}
Use the unitary decomposition of Lemma~\ref{lem:block-decomposition}.  On a scalar component belonging to $\ker G_h$, the exact relaxation matrix is zero, the time map is the identity, and the limit operator $E_hR(\Dt L_{D,h})P_h$ is also the identity.  On a component of $\ker G_h^*$, the limit operator vanishes and the full map is multiplication by $R(-\Dt/\delta)$.  Lemma~\ref{lem:rational-bounds} gives
\begin{equation}
 |R(-\Dt/\delta)|\le M\frac{\delta}{\Dt},
\end{equation}
for sufficiently small $\delta$.

Fix a nonzero singular value $\sigma_j$.  Since $B_{j,\delta}$ has two distinct eigenvalues for $\delta\le\delta_0$, functional calculus gives the exact identity
\begin{equation}
 R(\Dt B_{j,\delta})
 =R(\Dt\lambda_{s,j})\Pi_{s,j}
 +R(\Dt\lambda_{f,j})\Pi_{f,j}.
\label{eq:AP-block-functional}
\end{equation}
Let
\begin{equation}
 \Pi_{0,j}=\begin{bmatrix}1&0\\-a\sigma_j&0\end{bmatrix}.
\end{equation}
The eigenvalue and projector expansions imply, with constants uniform over the finite set of singular values,
\begin{equation}
 |\lambda_{s,j}+a\sigma_j^2|\le K_h\delta,
 \qquad \|\Pi_{s,j}-\Pi_{0,j}\|\le K_h\delta.
\label{eq:AP-slow-bounds}
\end{equation}
The slow arguments $\Dt\lambda_{s,j}$ remain in a fixed compact interval of the negative real axis.  If $L_R$ is a Lipschitz constant of $R$ on that interval, then
\begin{align}
&\left\|R(\Dt\lambda_{s,j})\Pi_{s,j}
 -R(-a\sigma_j^2\Dt)\Pi_{0,j}\right\|\nonumber\\
&\quad\le
 L_R\Dt|\lambda_{s,j}+a\sigma_j^2|\,\|\Pi_{s,j}\|
 +|R(-a\sigma_j^2\Dt)|\,\|\Pi_{s,j}-\Pi_{0,j}\|
 \le K_{h,\Dt}\delta.
\label{eq:AP-slow-error}
\end{align}
For the fast term, the eigenvalue formula gives $|\lambda_{f,j}|\ge c_h/\delta$ after reducing $\delta_0$ if necessary.  Therefore Lemma~\ref{lem:rational-bounds} and the uniform projector bound yield
\begin{equation}
 \|R(\Dt\lambda_{f,j})\Pi_{f,j}\|
 \le K_h\frac{\delta}{\Dt}.
\label{eq:AP-fast-error}
\end{equation}
The matrix $R(-a\sigma_j^2\Dt)\Pi_{0,j}$ is exactly the coupled-block representation of $E_hR(\Dt L_{D,h})P_h$.  Combining \eqref{eq:AP-block-functional}--\eqref{eq:AP-fast-error}, taking the maximum over the finite block set, and using invariance of the operator norm under the unitary basis transformation proves \eqref{eq:AP-estimate}.
\end{proof}

\begin{revision}
\begin{corollary}[The actual central-flux finite-volume step]\label{cor:fv-ap}
On a fixed periodic grid, use the fifth-order reconstruction \eqref{eq:fv-reconstruction}, the central face flux $\alpha=0$ in \eqref{eq:rusanov-specified}, and the face derivative in Algorithm~\ref{alg:trajectory-provider}. With exact stage solutions and an L-stable $C$, the complete implemented step obeys \eqref{eq:AP-estimate} with $G_h=D_h=T_h$ from \eqref{eq:central-fv-pair} and $L_{D,h}=aT_h^2$.
\end{corollary}
\begin{proof}
The face difference yields $T_h^*=-T_h$ and the residual equals \eqref{eq:APoperators} block for block. Proposition~\ref{prop:fv-closure} gives $\widetilde{\mathcal G}_h=\mathcal L_h^2$ for this linear flux. Theorem~\ref{thm:fullstep-ap} therefore applies to the same matrices used by the finite-volume stages.
\end{proof}
\end{revision}

\begin{corollary}[Fixed finite-time interval]
Let $n\le N_T$ be fixed independently of $\delta$, and let the initial data be uniformly bounded. Then
\begin{equation}
 \left\|R(\Dt A_{h,\delta})^n
 -\bigl(E_hR(\Dt L_{D,h})P_h\bigr)^n\right\|
 \le C_{h,\Dt,N_T}\delta.
\end{equation}
\end{corollary}
\begin{proof}
Set $X=R(\Dt A_{h,\delta})$ and $Y=E_hR(\Dt L_{D,h})P_h$.  The identity $P_hE_h=I$ implies
\begin{equation}
 Y^k=E_hR(\Dt L_{D,h})^kP_h,
\end{equation}
so A-stability of $R$ and negative semidefiniteness of $L_{D,h}$ give
$\|Y^k\|\le\|E_h\|\,\|P_h\|$ for all $k$.  Lemma~\ref{lem:uniform-step} gives $\|X^k\|\le M_{h,\Dt}^k$; because $k\le N_T$, these powers are uniformly bounded independently of $\delta$.  Apply
\begin{equation}
 X^n-Y^n=\sum_{k=0}^{n-1}X^{n-1-k}(X-Y)Y^k
\end{equation}
and use Theorem~\ref{thm:fullstep-ap} on $X-Y$.  The finite sum contains at most $N_T$ terms, which proves the stated $O(\delta)$ estimate.
\end{proof}

\begin{remark}[Scope of the theorem]
The constant may depend on $h$ and $\Dt$; the theorem establishes AP, not an $h$-uniform error bound. It is a full-step result because $H(z)\to1$ for $z\to-\infty$, so an unprepared fast component can remain at the half stage. The theorem also assumes exact compatible spatial operators and exact stage solves. It does not assert a general nonlinear-WENO AP theorem.
\end{remark}

\subsection{Perturbed derivative and inexact solve}
The exact assumptions clarify how implementation errors enter.
\begin{proposition}[Strict AP versus AP up to implementation tolerance]
Suppose a computed one-step map $\widetilde R_{h,\delta}$ satisfies
\begin{equation}
 \|\widetilde R_{h,\delta}-R(\Dt A_{h,\delta})\|\le\eta_{h,\delta}.
\end{equation}
Then
\begin{equation}
 \|\widetilde R_{h,\delta}-E_hR(\Dt L_{D,h})P_h\|
 \le K_{h,\Dt}\delta+\eta_{h,\delta}.
\label{eq:AP-perturbed}
\end{equation}
Therefore:
\begin{enumerate}[leftmargin=1.8em]
\item if $\eta_{h,\delta}=o(1)$ as $\delta\to0$, the computed method is strictly AP;
\item if $\eta_{h,\delta}=O(\delta)$, the $O(\delta)$ rate of Theorem~\ref{thm:fullstep-ap} is retained;
\item if $\eta_{h,\delta}\le\eta_{\rm tol}$ is fixed, then the method is AP only up to an $O(\eta_{\rm tol})$ implementation floor.
\end{enumerate}
\end{proposition}
\begin{proof}
Equation \eqref{eq:AP-perturbed} follows from the triangle inequality and Theorem~\ref{thm:fullstep-ap}.  The three conclusions follow by taking the corresponding limits.
\end{proof}
This distinction is important in strongly singular calculations: any approximation of $\GG_h$ and any Newton tolerance must be selected so that their one-step perturbation does not dominate either the temporal error or the desired AP rate.

\begin{revision}
Preparation-dependent accuracy across the initial layer and fourth-order recovery after it are analyzed in full in the independently compiled supplement. These results use the same fixed-grid linear slow--fast blocks but are not needed to establish the completed-step AP property or the asymptotic damping exponent. For arbitrary unprepared data, uniform full-state fourth order including the first step is generally false; after any fixed positive time the fast component is sufficiently damped to recover fourth order. This distinction limits how the field plots below should be read.
\end{revision}

\section{Implicit algebraic solves and computational structure}
\subsection{Linear stage matrices}
For $\LL_h(\UU)=A_h\UU$ and exact closure \eqref{eq:A2closure}, the two stages are
\begin{equation}
 \left[I-\frac{\Dt}{4}A_h+\frac{\Dt^2}{48}A_h^2\right]\UU^\star
 =\left[I+\frac{\Dt}{4}A_h+\frac{\Dt^2}{48}A_h^2\right]\UU^n,
\label{eq:linear-stage1}
\end{equation}
\begin{align}
 \left[I-a_2\Dt A_h+\frac32C\Dt^2A_h^2\right]\UU^{n+1}
={}&\left[I+a_0\Dt A_h+C\Dt^2A_h^2\right]\UU^n\nonumber\\
&+\left[a_1\Dt A_h-C\Dt^2A_h^2\right]\UU^\star.
\label{eq:linear-stage2}
\end{align}
The stage vectors are found sequentially. Writing $N_{\rm dof}$ for the total scalar unknowns, each stage has $N_{\rm dof}$ variables; a direct two-stage Gauss collocation formulation has $2N_{\rm dof}$ coupled stage variables. Block diagonalization or specialized transformations can reduce that cost, so sequentiality is an implementation advantage rather than a universal complexity theorem.

\begin{revision}The reference implementation factors the two matrices in \eqref{eq:linear-stage1}--\eqref{eq:linear-stage2}.\end{revision} Large-scale codes should apply $A_hv$ and $A_h(A_hv)$ matrix-free. For transport--relaxation systems, a natural preconditioner approximates
\begin{equation}
 P_1=I-\frac{\Dt}{4}\widetilde A_h+\frac{\Dt^2}{48}\widetilde A_h^2,
\qquad
 P_2=I-a_2\Dt\widetilde A_h+\frac32C\Dt^2\widetilde A_h^2.
\end{equation}
Eliminating the flux variable yields a diffusion-like Schur complement for $u$, while the local $1/\delta$ block can be inverted directly.

\subsection{Stage residuals and stopping tolerances}
\begin{revision}
The AP theorem assumes exact stages. The following calculation turns this into a checkable solver target. Put $A=A_{h,\delta}$, and denote the left matrices in \eqref{eq:linear-stage1}--\eqref{eq:linear-stage2} by $M_1,M_2$ and the midpoint coupling on their right by $K=a_1\Dt A-C\Dt^2A^2$. For approximate stages $Z,W$, define their \emph{actual} algebraic residuals $r_1=M_1Z-b_1$ and $r_2=M_2W-b_2-KZ$, where $b_1,b_2$ comprise only the known-state terms. Subtraction from the exactly solved stage equations gives
\begin{newrevision}
\begin{equation}
 W-\UU^{n+1}=M_2^{-1}r_2+M_2^{-1}K M_1^{-1}r_1
 =M_2^{-1}\bigl(r_2+K(M_1^{-1}r_1)\bigr).
\label{eq:solve-error-identity}
\end{equation}
\end{newrevision}
\begin{newrevision}The ordering is essential: $r_1$ is first acted on by $M_1^{-1}$, then by $K$, and finally by $M_2^{-1}$; no commutation of these operators is assumed.\end{newrevision}
Thus a sufficient, norm-aware per-step condition for fourth-order time accuracy on a fixed grid is
\begin{equation}
 \|M_2^{-1}\|_{h\to h}
 \bigl(\|r_2\|_h+\|K M_1^{-1}\|_{h\to h}\|r_1\|_h\bigr)
 \le c\,\Dt^5\,\|\UU^n\|_h .
\label{eq:solver-tolerance}
\end{equation}
Here $c$ is an accuracy budget set below the desired time-discretization error. For a fixed-$h$, fixed-$\Dt$ experiment intended to show the $O(\delta)$ AP slope, replace the right-hand budget by $b=c\min(\Dt^5,\delta)\|\UU^n\|_h$; otherwise a fixed absolute tolerance produces a floor as $\delta\to0$. This condition gives a concrete allocation. Given upper bounds $\gamma_2\ge\|M_2^{-1}\|_{h\to h}$ and $\gamma_1\ge\|KM_1^{-1}\|_{h\to h}$, require
\begin{equation}
 \|r_2\|_h\le\frac{b}{2\gamma_2},\qquad
 \|r_1\|_h\le\frac{b}{2\gamma_2\max(1,\gamma_1)}.
 \label{eq:linear-stage-budget}
\end{equation}
Substitution into \eqref{eq:solve-error-identity} gives a completed-step solve error at most $b$. \begin{newrevision}The two inequalities in \eqref{eq:linear-stage-budget} are sufficient, conservative residual criteria obtained by splitting the budget; they are not claimed to be sharp.\end{newrevision} For a fourth-order-only calculation take $b=c\Dt^5\|\UU^n\|_h$ instead; for zero data or an absolute error target, replace the state norm by a declared positive reference scale. A numerical gain estimate is useful for selecting tolerances, but a certified upper bound is needed to guarantee the inequality. The main finite-volume AP experiment computes exact rational stage matrices and has no iterative-solver tolerance.
\end{revision}

\begin{revision}
The linear stage identity \eqref{eq:solve-error-identity} supplies an exact gain-weighted residual budget. For the auxiliary nonlinear Jin--Xin example, the supplement derives a computable coupled-stage Newton--Kantorovich bound: after both solves, the achieved residual and Jacobian give a radius $\rho$ for the exact stage root whenever its majorant condition holds. The separate budget-enabled routine accepts a step only if this computed radius fits $c\min(\Dt^5,\delta)\max(1,\|U^n\|_\infty)$; otherwise it tightens Newton or reports that the budget cannot be attained at the available precision. The supplementary twelve-row table checks this rule against independent tighter solves at three positive $\delta$ values. Other supplementary sweeps use stated tolerances and are not certified by that acceptance routine. Numerical inverse norms are floating estimates, not interval-validated or uniform-in-$\delta$ guarantees.
\end{revision}

\subsection{Sequential stage implementation}
\begin{algorithm}[H]
\caption{\textcolor{black}{One complete implicit two-stage fourth-order step for the specified linear FV residual}}
\label{alg:step}
\begin{revision}
\begin{algorithmic}[1]
\State Evaluate $\LL_h(\UU^n)$ and $\GG_h(\UU^n)$.
\State For the specified linear FV residual, form $M_1,M_2,K,b_1,b_2$ from \eqref{eq:linear-stage1}--\eqref{eq:linear-stage2}.
\State Solve $M_1Z=b_1$ directly, or iteratively with actual residual $r_1=M_1Z-b_1$.
\State Solve $M_2W=b_2+KZ$ directly, or iteratively with actual residual $r_2=M_2W-b_2-KZ$.
\State If using iterations, check the two achieved residuals against \eqref{eq:linear-stage-budget}; set $\UU^{n+1}=W$.
\end{algorithmic}
\end{revision}
\end{algorithm}

\begin{revision}The two-stage count refers to unknown stage vectors rather than equal floating-point cost: evaluating $\GG_h$ can be more expensive than evaluating $\LL_h$. The first $N_{\rm dof}$-unknown system is completed before the second is formed, while standard two-stage Gauss collocation naturally couples two $N_{\rm dof}$-vectors. The linear implementation reuses constant matrix factors and the known-state data. For the ancillary nonlinear example, the same sequential order uses damped Newton solves and $2\UU^\star-\UU^n$ as a second-stage initial guess, with its separate stopping diagnostic in the supplement.\end{revision}

\section{One-dimensional numerical experiments}
Unless stated otherwise, the spatial grid is fixed and the reference is the exact semi-discrete solution, so that the reported errors isolate the temporal method.  The main experiments use double precision; the separate Prothero--Robinson experiment and its arithmetic are described only in the supplement.

\subsection{Coarser temporal steps and initial preparation}
\begin{revision}
To include genuinely coarse steps, we take one, two, four, eight, 16, or 32 steps with the \emph{same fifth-order finite-volume residual and independently assembled face derivative} used in the next subsection: $N=48$, $a=0.7$, $\delta=0.02$, and $t_f=0.2$. Initial data are $u_0=\sin x+0.2\cos2x$ with $q_0=-aT_hu_0$ (limiting equilibrium) or $q_0=0$ (unprepared). Each error is against the exact matrix exponential of this very same finite-volume residual at positive $\delta$. The steep first reduction crosses a fast transient and is not an asymptotic order estimate; for $\Delta t\le0.05$ the fixed-$\delta$ errors approach fourth order. The larger unprepared error shows the practical coarse-step effect. This experiment does not establish uniform accuracy across $\delta$.
\end{revision}
\begin{table}[H]
\centering
\caption{\textcolor{black}{Coarse-step error at $t_f=0.2$ for the central-flux finite-volume operator; the exact semi-discrete trajectories are the references.}}
\label{tab:coarse-steps}
\begin{tabular}{ccc}
\toprule
$\Delta t$ & $q_0=-aD_hu_0$ error & $q_0=0$ error\\
\midrule
0.20000 & $1.8240\times10^{-4}$ & $7.3202\times10^{-3}$\\
0.10000 & $5.2171\times10^{-7}$ & $1.7667\times10^{-5}$\\
0.05000 & $3.4713\times10^{-8}$ & $1.2167\times10^{-6}$\\
0.02500 & $2.0598\times10^{-9}$ & $7.2287\times10^{-8}$\\
0.01250 & $1.2689\times10^{-10}$ & $4.4552\times10^{-9}$\\
0.00625 & $7.9003\times10^{-12}$ & $2.7744\times10^{-10}$\\
\bottomrule
\end{tabular}
\end{table}
\begin{figure}[H]
\centering
\includegraphics[width=.70\linewidth,trim=0bp 7085bp 833bp 0bp,clip]{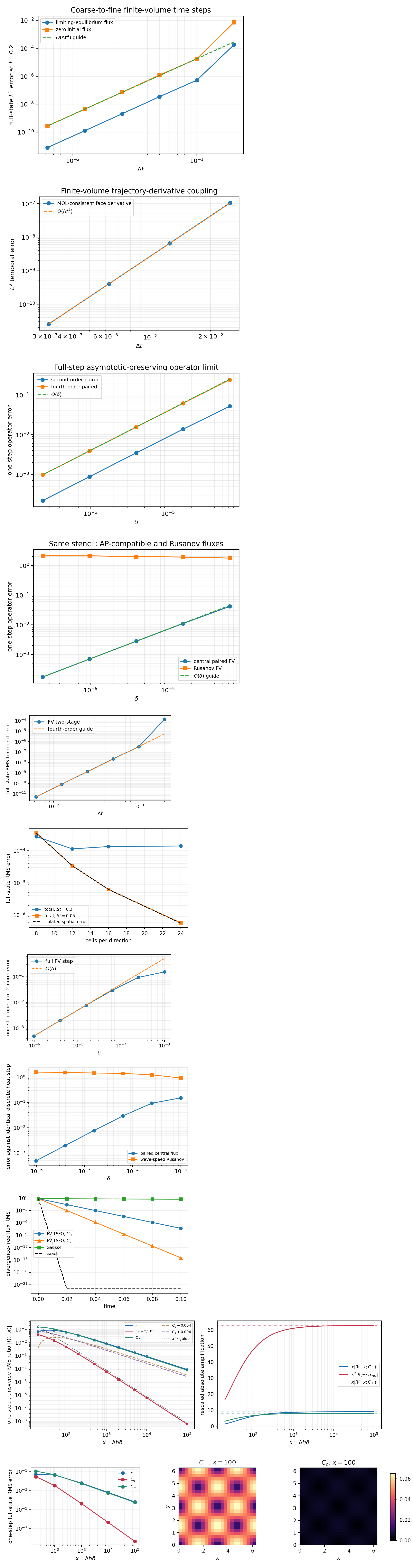}
\caption{\textcolor{black}{The actual finite-volume one-step to 32-step scan, including the coarse transition for prepared and unprepared data.}}
\label{fig:coarse-steps}
\end{figure}

\subsection{Coarse spatial grids with coarse time steps}
\begin{revision}
To expose spatial resolution as well as the time-step effect, Table~\ref{tab:coarse-mesh} repeats the \emph{same central-flux finite-volume residual and face-derivative stage operator} on $N=12,24,48$ cells at $\Delta t=0.2$ (one step) and $0.05$ (four steps). Here $u_0=\sin x+0.2\cos2x$ and $q_0=-a u_{0,x}$ are projected to exact cell averages; $a=0.7$, $\delta=0.02$, and $t_f=0.2$. One reference is the exact trajectory of each positive-$\delta$ finite-volume matrix. The second is the exact Fourier solution of the same positive-$\delta$ PDE, projected to cell averages. All errors use the root-mean-square norm of both fields. Thus the first error isolates temporal accuracy on a fixed grid; the second includes spatial discretization. In particular, refining the time step on $N=12$ leaves a spatial-error level near $4.8\times10^{-4}$, while the one-step error on $N=48$ is dominated by time discretization. These data are a finite-parameter diagnostic, not a joint uniform $h$--$\delta$--$\Delta t$ theorem.
\end{revision}
\begin{table}[tbp]
\centering
\caption{\textcolor{black}{Coarse grid and time-step errors for the actual central-flux finite-volume implementation.}}
\label{tab:coarse-mesh}
\begin{revision}\begin{tabular}{cccc}
\toprule
$N$ & $\Delta t$ & error vs. exact FV & error vs. PDE cell averages\\
\midrule
12 & 0.200 & $1.5823\times10^{-4}$ & $5.1243\times10^{-4}$\\
12 & 0.050 & $2.9185\times10^{-8}$ & $4.8064\times10^{-4}$\\
24 & 0.200 & $1.8043\times10^{-4}$ & $1.8125\times10^{-4}$\\
24 & 0.050 & $3.4294\times10^{-8}$ & $9.1258\times10^{-6}$\\
48 & 0.200 & $1.8199\times10^{-4}$ & $1.8200\times10^{-4}$\\
48 & 0.050 & $3.4627\times10^{-8}$ & $1.5322\times10^{-7}$\\
\bottomrule
\end{tabular}
\end{revision}
\end{table}

\subsection{Separating spatial and temporal error on coarse grids}
\begin{revision}
A small temporal error does not imply that the PDE solution is resolved. Let $U_h^{\rm num}$ be the computed vector, $U_h^{\rm sd}=\exp(t_fL_{h,\delta})U_h(0)$ the exact semi-discrete trajectory, and $\overline U^{\rm PDE}$ the exact cell averages of the positive-$\delta$ Fourier solution. In the same root-mean-square full-state norm,
\begin{equation}
 \|U_h^{\rm num}-\overline U^{\rm PDE}\|
 \le \underbrace{\|U_h^{\rm num}-U_h^{\rm sd}\|}_{E_t}
 +\underbrace{\|U_h^{\rm sd}-\overline U^{\rm PDE}\|}_{E_x}.
 \label{eq:error-decomposition-data}
\end{equation}
This is an upper bound rather than an equality, because vector errors may cancel. Table~\ref{tab:coarse-mesh} reports $E_t$ and total error; Table~\ref{tab:spatial-isolation} separately measures $E_x$ with the same initial cell averages and exact positive-$\delta$ references. Changing $\Delta t$ leaves $E_x$ unchanged. The ratios between meshes are finite-grid observations rather than an asymptotic order claim.
\end{revision}
\begin{table}[tbp]
\centering
\caption{\textcolor{black}{Spatial error isolated by the exact finite-volume matrix exponential at $\delta=0.02$ and $t_f=0.2$; generated by the experiment driver.}}
\label{tab:spatial-isolation}
\begin{revision}\begin{tabular}{cc}\toprule
$N$ & $E_x$ \\ \midrule
12 & $4.80642e-04$ \\ 
24 & $9.12607e-06$ \\ 
48 & $1.49676e-07$ \\ 
\bottomrule\end{tabular}
\end{revision}
\end{table}
\begin{revision}
At $N=12$, spatial error dominates after four steps; refining time alone cannot establish a fourth-order PDE error. At $N=48$ with one step, the temporal error dominates. This crossover explains why both reference solutions are needed, and why the fine-step plot below deliberately isolates time accuracy.
\end{revision}

\subsection{Finite-volume trajectory-derivative coupling}
The preceding coarse-step test already uses the face-based provider of Algorithm~\ref{alg:trajectory-provider}. Here the same smooth central-flux finite-volume residual on $48$ cells is tested with smaller $\Delta t$. The matrix $\widetilde G_h$ is assembled from interface flux-time derivatives according to \eqref{eq:linear-trajectory-flux}, independently of the residual matrix $L_h$. The stage solves use $(L_h,\widetilde G_h)$ directly. The matrix audit gives
\[
 \|\widetilde G_h-L_h^2\|_\infty=1.23\times10^{-12},\qquad
 \|T_h+T_h^T\|_\infty=0,\qquad
 \|L_h-A_{h,\delta}(T_h,T_h)\|_\infty<1.8\times10^{-15},
\]
where the last comparison is in cell-interleaved block ordering. Table~\ref{tab:fv-trajectory} and Figure~\ref{fig:fv-trajectory} show fourth-order temporal convergence to the exact semi-discrete solution. Thus the face derivative, paired spatial structure, and temporal stage matrices are verified on one implementation path.
\begin{table}[tbp]
\centering
\caption{Finite-volume time convergence with the independently assembled MOL-consistent face-derivative operator.}
\label{tab:fv-trajectory}
\begin{tabular}{ccc}\toprule
$\Delta t$ & $L^2$ error & order\\\midrule
0.025000 & $1.0573\times10^{-7}$ & --\\
0.012500 & $6.5248\times10^{-9}$ & 4.018\\
0.006250 & $4.0638\times10^{-10}$ & 4.005\\
0.003125 & $2.5378\times10^{-11}$ & 4.001\\\bottomrule
\end{tabular}
\end{table}
\begin{figure}[tbp]
\centering
\includegraphics[width=.58\linewidth,trim=0bp 6184bp 855bp 924bp,clip]{Manuscript_figures.png}
\caption{Fourth-order convergence when the stages use the face flux-time-derivative operator $\widetilde G_h$ rather than a preformed square of the residual matrix.}
\label{fig:fv-trajectory}
\end{figure}

\section{Direct verification of the full-step AP limit}
\subsection{Paired-operator comparison}
The one-step matrix $R(\Dt A_{h,\delta})$ is compared directly with
$E_hR(\Dt L_{D,h})P_h$ for second- and fourth-order paired operators. Figure~\ref{fig:AP} displays the operator-norm error. Each factor-of-four decrease in $\delta$ eventually produces a factor-of-four decrease in the error, as predicted by Theorem~\ref{thm:fullstep-ap}.
\begin{figure}[tbp]
\centering
\includegraphics[width=.63\linewidth,trim=0bp 5283bp 855bp 1825bp,clip]{Manuscript_figures.png}
\caption{$O(\delta)$ convergence of the full relaxation step to the equilibrium embedding of the discrete diffusion step.}
\label{fig:AP}
\end{figure}

The fourth-order paired operator has a larger fixed-$h$ constant in this particular diagnostic because its highest singular values are larger. This does not contradict the spatial convergence benefit of a high-order operator; Theorem~\ref{thm:fullstep-ap} allows the AP constant to depend on $h$ and on the discrete spectrum.

\subsection{The same finite-volume flux code in the AP limit}\label{sec:fv-ap-data}
\begin{revision}
We now test the \emph{actual} face-assembled fifth-order finite-volume matrices for both flux choices in \eqref{eq:rusanov-specified}. On a fixed $N=12$ periodic grid with $a=0.7$ and $\Delta t=0.03$, construct the exact matrix $L_{h,\delta}$ from the code's residual, and compare $R(\Delta t L_{h,\delta})$ to $E_hR(\Delta t aT_h^2)P_h$ in the Euclidean operator $2$-norm. The same $T_h$ and heat target are used for both fluxes. Only the central branch has the block form of Theorem~\ref{thm:fullstep-ap}; the Rusanov scalar-viscosity norm is also reported to make the excluded term visible.
\end{revision}
\begin{table}[tbp]
\centering\small
\caption{\textcolor{black}{One-step AP error of the actual finite-volume matrices; $\|B_{h,\delta}\|_\infty$ is the Rusanov upper-left scalar block. The central branch has $B_{h,\delta}=0$.}}
\label{tab:fv-ap}
\begin{tabular}{cccc}
\toprule
$\delta$ & central error & Rusanov error & Rusanov $\|B_{h,\delta}\|_\infty$\\
\midrule
$6.2500\times10^{-5}$ & $4.1538\times10^{-2}$ & $1.7570$ & $2.15595\times10^2$\\
$1.5625\times10^{-5}$ & $1.0984\times10^{-2}$ & $1.8990$ & $4.31190\times10^2$\\
$3.90625\times10^{-6}$ & $2.7849\times10^{-3}$ & $1.9789$ & $8.62381\times10^2$\\
$9.765625\times10^{-7}$ & $6.9867\times10^{-4}$ & $2.0924$ & $1.72476\times10^3$\\
$2.44140625\times10^{-7}$ & $1.7482\times10^{-4}$ & $2.1219$ & $3.44952\times10^3$\\
\bottomrule
\end{tabular}
\end{table}
\begin{figure}[tbp]
\centering
\includegraphics[width=.65\linewidth,trim=0bp 4382bp 855bp 2726bp,clip]{Manuscript_figures.png}
\caption{\textcolor{black}{With identical reconstruction and time integrator, the central FV branch approaches the discrete heat step at $O(\delta)$; the conventional Rusanov branch does not approach that target on this fixed grid.}}
\label{fig:fv-ap}
\end{figure}
\begin{revision}
Each factor-of-four reduction of $\delta$ eventually reduces the central error by about four. The Rusanov error instead increases over the tested range, while its scalar-viscosity norm doubles. Exact trajectory-derivative closure holds for \emph{both} fluxes, so the comparison isolates the extra spatial AP requirement. This is a fixed-grid result; $T_h$ has the usual central-derivative high-frequency kernel and no monotonicity claim is made for nonsmooth data.
\end{revision}

\subsection{Observed parameter ratios and interpretation}
\begin{revision}
The central errors in Table~\ref{tab:fv-ap} have successive ratios $3.78$, $3.94$, $3.99$, and $4.00$ as $\delta$ is divided by four. This agrees with a fixed-grid $O(\delta)$ slope over the displayed range but does not establish a mesh-independent constant. The Rusanov errors rise from $1.757$ to $2.122$. The checkerboard proof gives a nonzero limiting lower bound; the tabulated operator norm maximizes over all resolved modes and need not equal that bound. Rusanov's numerical-viscosity norm doubles whenever $\delta$ is divided by four, as its $\delta^{-1/2}$ scaling predicts. Both branches use the same stage derivative construction, so their contrast concerns the spatial residual.
\end{revision}

\section{Two spatial dimensions: the same Cattaneo model and a paired finite-volume extension}\label{sec:two-dimensional-main}
\begin{revision}
The reviewer's request to study one equation remains our organizing principle: in this section only the spatial dimension of the \emph{same linear Cattaneo relaxation law} changes. The domain is the periodic torus $\Omega=[0,2\pi]^2$ and the data are periodic cell averages. We consider
\begin{equation}
 u_t+\partial_xq_x+\partial_yq_y=0,\qquad
 \delta\partial_t\bm q+a\nabla u=-\bm q,
 \quad \bm q=(q_x,q_y)^T, \quad a,\delta>0.
 \label{eq:2d-continuum-new}
\end{equation}
Its exact $\delta=0$ constraint is $\bm q=-a\nabla u$ and its limiting equation is $u_t=a\Delta u$. The additional transverse flux component makes the dimensional extension a substantive test of fast-mode damping; no nonlinear or nonperiodic-boundary claim is inferred from it.

\subsection{Cartesian finite volumes and exact discrete face derivative}
Let $U_{ij}=(u_{ij},q_{x,ij},q_{y,ij})^T$ be the cell averages on an $N_x\times N_y$ uniform periodic grid, with $h_x=2\pi/N_x$ and $h_y=2\pi/N_y$. In $x$ and $y$, use the \emph{same} fifth-order linear cell-average reconstruction \eqref{eq:fv-reconstruction} along the corresponding row or column. Because the flux matrices are constant, the reconstructed row or column averages directly define face-averaged normal fluxes; no tangential quadrature is needed for this Cartesian tensor-product construction. At each face choose the central physical fluxes
\begin{equation}
 \widehat F_{i+1/2,j}=\tfrac12 A_x(U^-+U^+),\quad
 \widehat H_{i,j+1/2}=\tfrac12 A_y(U^-+U^+),\quad
 A_x=\begin{bmatrix}0&1&0\\a/\delta&0&0\\0&0&0\end{bmatrix},\quad
 A_y=\begin{bmatrix}0&0&1\\0&0&0\\a/\delta&0&0\end{bmatrix},
 \label{eq:2d-faces-new}
\end{equation}
with cell source $SU_{ij}$, $S=\mathrm{diag}(0,-1/\delta,-1/\delta)$. The residual is
\begin{equation}
 (\mathcal L_{h,\delta}^{(2)}U)_{ij}
 =-\frac{\widehat F_{i+1/2,j}-\widehat F_{i-1/2,j}}{h_x}
 -\frac{\widehat H_{i,j+1/2}-\widehat H_{i,j-1/2}}{h_y}+SU_{ij}.
 \label{eq:2d-fv-residual-new}
\end{equation}
An $x$-face flux is shared by the two neighboring cells in that direction, and likewise in $y$; periodic telescoping preserves the total scalar cell mass. Let $V=\mathcal L_{h,\delta}^{(2)}U$. Reconstruct $V$ with those identical row and column weights, apply the same $A_x,A_y$ face matrices, and difference the resulting face time derivatives and $SV$. Because all maps are linear and fixed,
\begin{equation}
 \widetilde{\mathcal G}_{h,\delta}^{(2)}(U)
 =\mathcal L_{h,\delta}^{(2)}(V)
 =\bigl(\mathcal L_{h,\delta}^{(2)}\bigr)^2U
 =D\mathcal L_{h,\delta}^{(2)}(U)\mathcal L_{h,\delta}^{(2)}(U).
 \label{eq:2d-closure-new}
\end{equation}
This equality concerns the actual cell-residual trajectory; the face derivative is assembled separately in the code. It does not identify a physical CK derivative or a local space--time ADER predictor with the discrete trajectory derivative.

Write $T_{x,h}=I_y\otimes T_{h_x}$ and $T_{y,h}=T_{h_y}\otimes I_x$, where each $T$ has the seven-point coefficients of \eqref{eq:central-fv-pair}, using its respective mesh width. Both are skew-adjoint under $h_xh_y\sum_{ij}$ and commute. Thus, with
\begin{equation}
 G_hu=(T_{x,h}u,T_{y,h}u)^T,\quad
 D_h(q_x,q_y)=T_{x,h}q_x+T_{y,h}q_y,\quad
 D_h=-G_h^*,
 \label{eq:2d-pair-new}
\end{equation}
we obtain precisely
\begin{equation}
 A_{h,\delta}^{(2)}=\begin{bmatrix}
 0&-T_{x,h}&-T_{y,h}\\
 -aT_{x,h}/\delta&-I/\delta&0\\
 -aT_{y,h}/\delta&0&-I/\delta
 \end{bmatrix},\qquad
 L_{D,h}^{(2)}=a(T_{x,h}^2+T_{y,h}^2)=-aG_h^*G_h.
 \label{eq:2d-operator-new}
\end{equation}
The discrete energy $[a\|u\|_h^2+\delta\|\bm q\|_h^2]/2$ has derivative $-\|\bm q\|_h^2$, by the same discrete integration by parts as Lemma~\ref{lem:energy}. This establishes the required paired structure for the \emph{implemented} Cartesian flux, rather than for an unrelated Fourier surrogate.
\end{revision}

\begin{theorem}[Exact face closure and full-step AP in two dimensions]\label{thm:2d-ap-new}
\begin{revision}
\begin{newrevision}For the linear constant-coefficient Cartesian relaxation system, \end{newrevision}on each fixed $N_x\times N_y$ periodic grid, let the stages \eqref{eq:stage1}--\eqref{eq:stage2} use \eqref{eq:2d-fv-residual-new} and its face derivative \eqref{eq:2d-closure-new}, and be solved exactly. For fixed $\Dt>0$ and any admissible L-stable $C$, the complete one-step map satisfies
\begin{equation}
 \big\|R(\Delta t A_{h,\delta}^{(2)})
 -E_h^{(2)}R(\Delta t L_{D,h}^{(2)})P_h^{(2)}\big\|_{h\to h}
 \le K_{h,\Delta t}\delta,
 \quad E_h^{(2)}u=(u,-aT_{x,h}u,-aT_{y,h}u),\quad P_h^{(2)}(u,\bm q)=u,
 \label{eq:2d-ap-estimate-new}
\end{equation}
for all sufficiently small $\delta$. At finite fixed final time the bound extends to a fixed number of completed steps. The constant need not remain bounded as the grid is refined. \begin{newrevision}No mesh-uniform or nonlinear AP conclusion is intended.\end{newrevision}
\end{revision}
\end{theorem}
\begin{proof}
\begin{revision}
The face differentiation identity \eqref{eq:2d-closure-new} gives the exact rational matrix map $R(\Delta t A_{h,\delta}^{(2)})$ after both stages. Take a singular-value decomposition $G_h=U\Sigma V^*$ in the fixed $h_xh_y$-weighted spaces. For every positive singular value $\sigma_j$, the scalar coefficient $\widehat u_j$ and the longitudinal flux coefficient $\widehat q_j$ satisfy
\begin{equation}
 \frac{\mathrm d}{\mathrm dt}\begin{bmatrix}\widehat u_j\\\widehat q_j\end{bmatrix}
 =\begin{bmatrix}0&\sigma_j\\-a\sigma_j/\delta&-1/\delta\end{bmatrix}
 \begin{bmatrix}\widehat u_j\\\widehat q_j\end{bmatrix}.
 \label{eq:2d-longitudinal-block-new}
\end{equation}
The block is identical to \eqref{eq:Bblock}. Its slow eigenvalue differs from $-a\sigma_j^2$ by $O(\delta)$, its slow projector differs from $\left[\begin{smallmatrix}1&0\\-a\sigma_j&0\end{smallmatrix}\right]$ by $O(\delta)$, and its fast completed-step factor is $R(\Delta t\lambda_{f,j})=O(\delta/\Delta t)$. The Lipschitz estimate for $R$ on the bounded slow spectrum therefore gives an $O(\delta)$ difference from the embedded heat step on this block, exactly as in \eqref{eq:AP-slow-error}--\eqref{eq:AP-fast-error}. A scalar vector in $\ker G_h$ is stationary and reproduced exactly by the heat step. A flux vector in $\ker G_h^*$ is orthogonal to all longitudinal components, obeys $q_t=-q/\delta$, and contributes $R(-\Delta t/\delta)=O(\delta/\Delta t)$ to the numerical map while the heat target annihilates it. In two dimensions this kernel contains the transverse divergence-free flux for a nonzero wavevector, so this last check is essential. The orthogonal sum has finitely many blocks at fixed grid size, giving one constant $K_{h,\Delta t}$ in \eqref{eq:2d-ap-estimate-new}. The fixed-step finite-time result follows by telescoping the difference of powers, as in the corollary to Theorem~\ref{thm:fullstep-ap}. The half-stage is excluded because $H(z)\to1$ for a deeply negative fast argument.
\end{revision}
\end{proof}

\begin{proposition}[\textcolor{red!75!black}{Wave-speed Rusanov is not fixed-grid AP for the specified central-flux discrete heat target in two dimensions}]\label{prop:2d-rusanov-new}
\begin{revision}
Replace each central face flux in \eqref{eq:2d-faces-new} by
$\widehat F^{\rm R}=\widehat F-\alpha(U^+-U^-)/2$ and similarly for
$\widehat H^{\rm R}$, with $\alpha=\sqrt{a/\delta}$. Its face time
 derivative still equals the derivative of its own linear residual, but
on an even periodic Cartesian grid the complete step does not converge to
the central discrete heat target in \eqref{eq:2d-ap-estimate-new}. \begin{newrevision}This statement is confined to the displayed central target and does not apply to Rusanov-type discretizations in general.\end{newrevision}
\end{revision}
\end{proposition}
\begin{proof}
\begin{revision}
Take $u_{ij}=(-1)^i$, constant in $y$, and $q_x=q_y=0$. The central
 derivatives $T_{x,h}u=T_{y,h}u=0$, so the heat target leaves this
 mode unchanged. The $x$-face scalar reconstructions are $(8/15)u_{ij}$
and $-(8/15)u_{ij}$; their central average vanishes. The wave-speed
 viscosity gives $u_t=-16\alpha u/(15h_x)$ while all other couplings
 and the $y$-face jump vanish. Its complete scalar amplification is
 $R(-16\alpha\Delta t/(15h_x))\to0$ as $\delta\to0$ by L-stability,
 whereas the target amplification is one. Hence the operator-norm error
 has a nonzero limiting lower bound. Exact trajectory closure alone does
 not supply the spatial AP pairing. This grid-scale counterexample is
 specific to the fixed-grid operator norm and this central discrete heat
 target; it is not a smooth-mode accuracy comparison or a universal
 impossibility result for all dissipative flux designs.
\end{revision}
\end{proof}

\begin{revision}
The supplementary fixed-grid analysis also covers the additional transverse flux: exact slow preparation or observation after any fixed positive time gives uniform fourth-order temporal accuracy, while a nonzero divergence-free initial flux generally prevents uniform full-state convergence at the first step. Thus the Gaussian field experiment below is a temporal diagnostic for its stated parameters, not a uniform initial-layer theorem.
\end{revision}

\FloatBarrier
\subsection{Numerical design and independently assembled audits}
\begin{revision}
Except for the explicitly identified Rusanov comparison, all following measurements use the \emph{same face-assembled Cartesian central FV residual} and the same derivative provider \eqref{eq:2d-closure-new}. Fourier diagonalization is used only to solve this translation-invariant matrix efficiently and to compute its exact semi-discrete exponential; its symbol is obtained from the actual face stencil,
\begin{equation}
 \widehat T_{h}(\theta)=\frac{i}{30h}\bigl(45\sin\theta-9\sin2\theta+\sin3\theta\bigr).
 \label{eq:2d-symbol-new}
\end{equation}
On a reproducible random $12\times10$ state with $a=0.5,\delta=0.02$, the maximum cellwise discrepancy between the face residual and this Fourier matrix is $5.68\times10^{-14}$; the discrepancy between the separately face-assembled derivative and a second residual application is zero in the recorded double-precision calculation. The equality is audited on the implemented face residual, without a separate physical CK predictor.
\end{revision}

\begin{table}[tbp]\centering\small
\caption{\textcolor{black}{Two-dimensional temporal error of the fifth-order central FV operator, $16\times16$ cells, $a=0.5,\delta=0.02,t_f=0.2$. The reference is its exact semi-discrete trajectory; the error is RMS over all three fields.}}
\label{tab:2d-temporal-new}\begin{revision}\begin{tabular}{ccc}\toprule
$\Delta t$ & $E_t$ & order \\ \midrule
0.200000 & 1.4007e-04 & -- \\
0.100000 & 3.3130e-07 & 8.72 \\
0.050000 & 2.2935e-08 & 3.85 \\
0.025000 & 1.3619e-09 & 4.07 \\
0.012500 & 8.3925e-11 & 4.02 \\
0.006250 & 5.2245e-12 & 4.01 \\
\bottomrule\end{tabular}
\end{revision}\end{table}
\begin{figure}[tbp]\centering\includegraphics[width=.55\linewidth,trim=0bp 3806bp 1209bp 3627bp,clip]{Manuscript_figures.png}
\caption{\textcolor{black}{Coarse-to-fine temporal errors for the actual Cartesian FV matrix.}}\label{fig:2d-temporal-new}\end{figure}
\begin{revision}
Table~\ref{tab:2d-temporal-new} starts at a single $0.2$ step. The first apparent order of $8.72$ reflects passage across a resolved fast transient and must not be interpreted as an eighth-order scheme. Starting from four steps the observed rates approach four. Because the exact \emph{same} semi-discrete operator is the reference, this experiment isolates time error.
\end{revision}
\begin{table}[tbp]\centering\small
\caption{\textcolor{black}{Coarse two-dimensional FV grids: $E_t$ versus exact FV evolution, $E_x$ between exact FV evolution and exact positive-$\delta$ PDE cell averages, and total error at $t_f=0.2$ ($a=0.5,\delta=0.02$).}}
\label{tab:2d-coarse-new}\begin{revision}\begin{tabular}{ccccc}\toprule
$N_x=N_y$ & $\Delta t$ & $E_t$ & $E_x$ & $E_{total}$ \\ \midrule
8 & 0.20 & 1.23e-04 & 3.53e-04 & 2.73e-04 \\
8 & 0.05 & 2.01e-08 & 3.53e-04 & 3.53e-04 \\
12 & 0.20 & 1.36e-04 & 3.40e-05 & 1.13e-04 \\
12 & 0.05 & 2.23e-08 & 3.40e-05 & 3.41e-05 \\
16 & 0.20 & 1.38e-04 & 6.26e-06 & 1.34e-04 \\
16 & 0.05 & 2.26e-08 & 6.26e-06 & 6.28e-06 \\
24 & 0.20 & 1.39e-04 & 5.63e-07 & 1.39e-04 \\
24 & 0.05 & 2.28e-08 & 5.63e-07 & 5.80e-07 \\
\bottomrule\end{tabular}
\end{revision}\end{table}
\begin{revision}
For Table~\ref{tab:2d-coarse-new}, the initial Fourier mode $u=\sin x\cos y$, $\bm q=-a\nabla u$ is projected to \emph{exact} cell averages. An exact Fourier solution of the positive-$\delta$ PDE is projected to the same cell averages. Write $U_h^{\rm num}$ for the computed cell field, $U_h^{\rm sd}$ for the exact semi-discrete FV trajectory, and $\overline U^{\rm PDE}$ for those exact positive-$\delta$ PDE cell averages. The columns use the same three-field RMS norm and report
\begin{equation}
 E_t=\|U_h^{\rm num}-U_h^{\rm sd}\|,\qquad
 E_x=\|U_h^{\rm sd}-\overline U^{\rm PDE}\|,\qquad
 E_{\rm total}=\|U_h^{\rm num}-\overline U^{\rm PDE}\|\le E_t+E_x.
 \label{eq:2d-error-split-new}
\end{equation}
Cancellation can make the total error smaller than either constituent. At $8\times8$ cells the four-step error is already controlled by $E_x=3.53\times10^{-4}$, whereas on $24\times24$ cells a single coarse time step dominates. The spatial error falls across the four displayed meshes, but these finite-grid values do not prove a mesh-uniform AP or nonsmooth-data result.
\end{revision}

\begin{figure}[tbp]\centering\includegraphics[width=.55\linewidth,trim=0bp 3160bp 1121bp 4203bp,clip]{Manuscript_figures.png}
\caption{\textcolor{black}{Coarse-grid crossover: four-step total error approaches the measured spatial floor, while single-step error dominates on the finer grids.}}
\label{fig:2d-crossover-new}\end{figure}

\begin{table}[tbp]\centering\small
\caption{\textcolor{black}{Two-dimensional full-step AP operator comparison for the same central FV stencil, $8\times8$ cells, $\Delta t=0.03$, $a=0.5$; each ratio compares errors at consecutive factor-four reductions of $\delta$.}}
\label{tab:2d-ap-new}\begin{revision}\begin{tabular}{ccc}\toprule
$\delta$ & $\|R(\Delta t A)-ER(\Delta t L_D)P\|_2$ & ratio \\ \midrule
1.000e-03 & 1.5273e-01 & -- \\
2.500e-04 & 9.3534e-02 & 1.63 \\
6.250e-05 & 2.9374e-02 & 3.18 \\
1.563e-05 & 7.7760e-03 & 3.78 \\
3.906e-06 & 1.9720e-03 & 3.94 \\
9.766e-07 & 4.9478e-04 & 3.99 \\
\bottomrule\end{tabular}
\end{revision}\end{table}
\begin{figure}[tbp]\centering\includegraphics[width=.55\linewidth,trim=0bp 2584bp 1209bp 4849bp,clip]{Manuscript_figures.png}
\caption{\textcolor{black}{Convergence of the \emph{complete} two-dimensional FV time-step matrix to its discrete heat step.}}\label{fig:2d-ap-new}\end{figure}
\begin{revision}
On a periodic grid, the unitary two-dimensional DFT block-diagonalizes the FV matrix. For each $3\times3$ complex mode we compute the spectral norm of the difference between the full rational step and $E_h^{(2)}R(\Delta t L_{D,h}^{(2)})P_h^{(2)}$, and take the largest norm over modes. This equals the global operator $2$-norm. In Table~\ref{tab:2d-ap-new}, the first factor-four $\delta$ change reduces error by only $1.63$; the last reduces it by $3.99$. The pre-asymptotic transition is displayed instead of being removed to make the $O(\delta)$ range look wider. Grid and time step remain fixed throughout.
\end{revision}

\begin{table}[tbp]\centering\small
\caption{\textcolor{black}{Two-dimensional central versus wave-speed Rusanov fluxes with identical reconstruction, integrator, $8\times8$ grid and central discrete heat target. The final column is the scalar-viscosity block norm.}}
\label{tab:2d-flux-contrast-new}\begin{revision}\begin{tabular}{cccc}\toprule
$\delta$ & central error & Rusanov error & $\|B_{h,\delta}\|_2$ \\ \midrule
1.000e-03 & 1.527e-01 & 9.330e-01 & 6.074e+01 \\
2.500e-04 & 9.353e-02 & 1.251e+00 & 1.215e+02 \\
6.250e-05 & 2.937e-02 & 1.419e+00 & 2.429e+02 \\
1.563e-05 & 7.776e-03 & 1.475e+00 & 4.859e+02 \\
3.906e-06 & 1.972e-03 & 1.556e+00 & 9.718e+02 \\
9.766e-07 & 4.948e-04 & 1.598e+00 & 1.944e+03 \\
\bottomrule\end{tabular}
\end{revision}\end{table}
\begin{figure}[tbp]\centering\includegraphics[width=.56\linewidth,trim=0bp 1938bp 1121bp 5425bp,clip]{Manuscript_figures.png}
\caption{\textcolor{black}{The central FV branch approaches the heat step while the wave-speed Rusanov branch does not, although both differentiate their own face residual exactly.}}
\label{fig:2d-flux-contrast-new}\end{figure}
\begin{revision}
The explicit checkerboard proof predicts a nonvanishing Rusanov error. Table~\ref{tab:2d-flux-contrast-new} and Figure~\ref{fig:2d-flux-contrast-new} measure the implemented complete-step matrices: as $\delta$ falls from $10^{-3}$ to $9.77\times10^{-7}$, central error falls from $0.153$ to $4.95\times10^{-4}$, but Rusanov error rises from $0.933$ to $1.598$. Its viscosity norm doubles under each factor-four reduction in $\delta$. These operator norms maximize over all resolved modes and need not equal the checkerboard lower bound. The comparison isolates the extra spatial AP condition from exact face-derivative closure.
\end{revision}

\subsection{Localized field and independent transverse relaxation}
\begin{table}[tbp]\centering\small
\caption{\textcolor{black}{Two-dimensional Gaussian at $48\times48$, $a=0.5,\delta=10^{-4}$, $t_f=0.1$: five complete steps of $\Delta t=0.02$ versus the exact \emph{same} semi-discrete FV trajectory. The last entry is mean-scalar mass drift.}}
\label{tab:2d-gaussian-new}\begin{revision}\begin{tabular}{ccccc}\toprule
$L^1(u)$ & $L^2(u)$ & $L^\infty(u)$ & $L^2(\mathbf q)$ & mass drift \\ \midrule
4.750e-09 & 1.788e-08 & 2.400e-07 & 6.705e-08 & 3.469e-18 \\
\bottomrule\end{tabular}
\end{revision}\end{table}
\begin{revision}
The Gaussian starts with $q_x=q_y=0$ and is therefore initially unprepared. Table~\ref{tab:2d-gaussian-new} reports scalar and flux norms and mass drift; the FV field image in the supplement indicates where the small temporal error lies. This comparison uses identical sampled initial arrays on both trajectories, so it cannot measure the continuum PDE spatial error. It neither proves uniform initial-layer accuracy nor extends Theorem~\ref{thm:2d-ap-new} to a varying mesh.

Finally set a discrete streamfunction $\psi_{ij}=\sin x_i\sin y_j$ and initialize $u=0$, $q_x=T_{y,h}\psi$, $q_y=-T_{x,h}\psi$. Commutativity gives $D_h\bm q=0$ \emph{exactly}. Its scalar stays zero and the flux solves $\bm q_t=-\bm q/\delta$; this directly tests the transverse block omitted from any one-dimensional plot. On $24\times24$ cells with $\delta=10^{-4}$ and five steps of $\Delta t=0.02$, the initial RMS flux is $0.707105$. The computed TSFO4 final RMS flux is $4.5744\times10^{-8}$, matching $\|q_0\||R(-200)|^5$ to roundoff. The fourth-order Gauss scalar factor would retain RMS flux $0.523836$; this contrast illustrates L-stable damping, while the precise numbers depend on these time steps and data. The computed maximum scalar magnitude is $1.12\times10^{-17}$. The supplement labels the FV Gaussian image separately from the older fourth-order centered-difference diagnostics; none of those older diagnostics is substituted for the FV results in this section.
\end{revision}
\begin{figure}[tbp]\centering\includegraphics[width=.56\linewidth,trim=0bp 1292bp 1121bp 6071bp,clip]{Manuscript_figures.png}
\caption{\textcolor{black}{Computed transverse FV-flux norms for $C_+$ and $C_q$, compared with the fourth-order Gauss factor and exact relaxation. Tiny exact values are clipped only for logarithmic display.}}
\label{fig:2d-transverse-new}\end{figure}
\begin{revision}
Figure~\ref{fig:2d-transverse-new} shows the actual five-step FV history rather than only the final scalar factor. At $t_f=0.1$, the $C_q=5/183$ branch reaches $3.36\times10^{-15}$ in flux RMS versus $4.57\times10^{-8}$ for $C_+$ and $0.524$ for Gauss4. This is a two-dimensional consequence of the already derived quadratic deep-stiff decay, with its scope limited to this transverse mode and stiffness ratio.
\end{revision}

\FloatBarrier
\subsection{The unique quadratic-decay parameter across deep-stiff scales}
\label{sec:quadratic-damping-fv}
\begin{revision}
Corollary~\ref{cor:unique-quadratic-damping} is an asymptotic statement about
\emph{one fixed} coefficient. A five-step test at one stiffness ratio cannot
distinguish a transient amplitude advantage from a different decay exponent.
We therefore keep the same $24\times24$ periodic Cartesian central-FV grid,
$a=0.5$, and $\Delta t=0.02$, and vary $x=\Delta t/\delta$ from $20$ to
$10^5$. At every $x$, initialize the identical discrete transverse flux
$q_x=T_{y,h}\psi$, $q_y=-T_{x,h}\psi$ and $u=0$, where
$\psi_{ij}=\sin x_i\sin y_j$. Because $D_h\bm q=0$ exactly, the
one-step full-state RMS ratio is $|R(-x;C)|$: the spatial wave number
cannot change this fast eigenvalue. We nevertheless compute each marker by
\emph{applying the full three-field Cartesian FV time-step matrix} to the
initial arrays, independently of the scalar rational evaluation.

Table~\ref{tab:quadratic-damping-fv} reports those computed ratios at four
decade-spaced stiffness levels. Its $p_q$ is the log slope between adjacent
table rows: $p_q=\log(A_q(x_1)/A_q(x_2))/\log(x_2/x_1)$, where $A_q$ is
the computed FV RMS ratio. From $x=10^2$ to $10^5$, the $C_q$ ratio
decreases from $4.760\times10^{-3}$ to $6.278\times10^{-9}$, with the
last local exponent $1.999$. At $10^5$, $C_-$ and $C_+$ instead retain
$9.036\times10^{-5}$ and $8.038\times10^{-5}$, respectively. The
maximum relative FV-versus-80-digit scalar-factor discrepancy across all
thirty $x,C$ evaluations is $3.1\times10^{-9}$; the largest computed
scalar-component leakage is below $1.3\times10^{-14}$.
\end{revision}
\begin{table}[tbp]\centering\small
\caption{\textcolor{black}{One completed time-step RMS amplification of the actual two-dimensional divergence-free central-FV flux, $24\times24$ cells, $\Delta t=0.02$ and $\delta=\Delta t/x$. All columns use the same initial flux; $p_q$ is a decade-to-decade exponent for $C_q$.}}
\label{tab:quadratic-damping-fv}\begin{revision}\begin{tabular}{ccccc}\toprule
$x=\Delta t/\delta$ & FV $C_-$ & FV $C_q$ & FV $C_+$ & local $p_q$ \\ \midrule
1e+02 & 6.214e-02 & 4.760e-03 & 6.628e-02 & -- \\
1e+03 & 8.706e-03 & 6.107e-05 & 7.885e-03 & 1.892 \\
1e+04 & 9.005e-04 & 6.262e-07 & 8.024e-04 & 1.989 \\
1e+05 & 9.036e-05 & 6.278e-09 & 8.038e-05 & 1.999 \\
\bottomrule\end{tabular}
\end{revision}
\end{table}
\begin{figure}[tbp]\centering
\includegraphics[width=.96\linewidth,trim=0bp 589bp 133bp 6717bp,clip]{Manuscript_figures.png}
\caption{\textcolor{black}{Deep-stiff parameter sweep. Left: high-precision scalar curves for $C_-,C_q,C_+$ and the neighboring fixed values $C_q\pm0.004$ (dashed); markers are independent complete-step Cartesian-FV transverse-flux ratios for the three main values. Right: rescaled factors approach $|-61/3+5/(9C_\pm)|$ and $314/5$, respectively (dotted horizontal limits). The guide indicates inverse-square decay, not a fitted order.}}
\label{fig:quadratic-damping-fv}
\end{figure}
\begin{revision}
Figure~\ref{fig:quadratic-damping-fv} makes both parts of the claim visible:
the neighboring \emph{fixed} parameters eventually return to a first
inverse power, while $x^2|R(-x;C_q)|$ tends to $314/5$ and
$x|R(-x;C_\pm)|$ tends to its nonzero predicted coefficient. Sampling
cannot prove uniqueness; Corollary~\ref{cor:unique-quadratic-damping} does.
For one prescribed finite $x$, the mode-specific $C_x$ in
\eqref{eq:finite-x-zero} can give an even smaller scalar factor, but it
changes with $x$ and need not optimize a solution containing both slow and
fast modes. This experiment establishes the strongest \emph{asymptotic
damping order within the fixed-parameter family}, not an unconditional
accuracy or runtime optimum.
\end{revision}

\begin{revision}
To test whether this mode-wise gain remains visible in a complete two-dimensional
state, add a longitudinal field $u=\sin x\cos y$ with the discrete-equilibrium
flux $\bm q_{\rm long}=-aG_hu$ to the same transverse flux
$\bm q_{\rm tr}=(T_{y,h}\psi,-T_{x,h}\psi)$, where
$\psi=\sin x\sin y$. This initial state has a nonzero slow diffusive
component and a nonzero fast component; only the latter is exactly transverse.
At the same $24\times24$ grid and $\Delta t=0.02$, compare one complete
FV step against $e^{\Delta t A_{h,\delta}^{(2)}}U_0$, the exact
\emph{positive-$\delta$ semi-discrete} trajectory from identical arrays.
Set $E_{\rm full}(C)=\|U_{\rm FV}^{1}(C)-e^{\Delta t A_{h,\delta}^{(2)}}U_0\|_{\rm RMS}$.
Table~\ref{tab:quadratic-mixed-fv} gives the full-state errors. At $x=100$,
$C_q$ gives $3.366\times10^{-3}$ against $4.686\times10^{-2}$ for
$C_+$, a factor of $13.9$ for this prescribed field. At $x=10^5$ the
corresponding errors are $4.439\times10^{-9}$ and
$5.683\times10^{-5}$. Figure~\ref{fig:quadratic-mixed-fv} shows both
the sweep and flux-vector error maps at $x=100$ on a common color scale.
The scalar component of this deliberately fast-rich field has much smaller
errors (at $x=100$, $9.47\times10^{-11}$ for $C_q$); thus the full-state
advantage is driven by the fast flux, not by a claim of universally better
smooth-mode accuracy. This one-step exact-FV comparison does not quantify
continuum spatial error or runtime, and the large-$x$ scalar errors approach
floating-point precision.
\end{revision}
\begin{table}[H]\centering\small
\caption{\textcolor{black}{One-step full-state RMS error of a mixed longitudinal plus transverse two-dimensional FV field against the exact positive-$\delta$ semi-discrete trajectory; $24\times24$ cells, $a=0.5$, $\Delta t=0.02$, $\delta=\Delta t/x$. The final column is the improvement for this specified field.}}
\label{tab:quadratic-mixed-fv}\begin{revision}\begin{tabular}{ccccc}\toprule
$x=\Delta t/\delta$ & $E_{\rm full}(C_-)$ & $E_{\rm full}(C_q)$ & $E_{\rm full}(C_+)$ & $E_{\rm full}(C_+)/E_{\rm full}(C_q)$ \\ \midrule
2e+01 & 4.936e-02 & 2.912e-02 & 1.102e-01 & 3.8 \\
1e+02 & 4.394e-02 & 3.366e-03 & 4.686e-02 & 13.9 \\
1e+03 & 6.156e-03 & 4.318e-05 & 5.575e-03 & 129.1 \\
1e+04 & 6.368e-04 & 4.428e-07 & 5.674e-04 & 1281.2 \\
1e+05 & 6.389e-05 & 4.439e-09 & 5.683e-05 & 12802.3 \\
\bottomrule\end{tabular}
\end{revision}
\end{table}
\begin{figure}[H]\centering
\includegraphics[width=.98\linewidth,trim=0bp 0bp 0bp 7420bp,clip]{Manuscript_figures.png}
\caption{\textcolor{black}{Mixed-field one-step error. Left: full-state RMS error against the exact semi-discrete trajectory. Center/right: flux-vector error for $C_+$ and $C_q$ at $x=100$ on an identical color scale. The spatial pattern and full-state norm reveal the fast-flux improvement for these data, not a general optimality result for slow accuracy.}}
\label{fig:quadratic-mixed-fv}
\end{figure}

\begin{revision}
The numerical evidence has distinct roles: Table~\ref{tab:2d-temporal-new} and Figure~\ref{fig:2d-temporal-new} check fourth-order time accuracy on the actual FV matrix; Table~\ref{tab:2d-coarse-new} and Figure~\ref{fig:2d-crossover-new} separate spatial from temporal error; Table~\ref{tab:2d-ap-new} and Figure~\ref{fig:2d-ap-new} test the complete-step AP limit; Table~\ref{tab:2d-flux-contrast-new} and Figure~\ref{fig:2d-flux-contrast-new} show that derivative closure alone is insufficient for this heat target; Figure~\ref{fig:2d-transverse-new} checks the independent fast mode; Table~\ref{tab:quadratic-damping-fv} and Figure~\ref{fig:quadratic-damping-fv} test the unique quadratic-decay coefficient over varying stiffness; Table~\ref{tab:quadratic-mixed-fv} and Figure~\ref{fig:quadratic-mixed-fv} check its full-state impact for one specified mixed field. The Gaussian is a field audit with measured error norms. \begin{newrevision}These experiments do not measure wall-clock efficiency.\end{newrevision}
\end{revision}

\section{Conclusions and future work}
\begin{newrevision}
The temporal formula, its fourth-order conditions, and the L-stable parameter interval are taken from \cite{Huo2026}; the discrete MOL differentiation principle follows \cite{SchutzSealJaust2017}. The contribution here is the conservative face realization and adjoint-pair verification for the specified finite-volume residual, together with a full-step AP proof for that same central-flux implementation in one and two dimensions and direct numerical audits. The linear closure identity is an implementation condition, not a new general differentiation principle.
\end{newrevision}

\begin{newrevision}
Within the inherited temporal family, the parameter analysis identifies $C_q=5/183$ as the unique fixed admissible coefficient with quadratic asymptotic inverse-stiffness decay. The two-dimensional transverse sweep confirms that exponent, and the mixed slow--fast test shows a smaller full-state one-step error for the stated data. These observations concern the asymptotic damping exponent and the reported test fields; they do not establish an optimum for finite-stiffness accuracy, overall accuracy, runtime, or computational efficiency. The Rusanov counterexample concerns only the specified central-flux discrete heat target on a fixed grid and is not a general negative result for Rusanov-type discretizations.
\end{newrevision}

\begin{newrevision}
All AP statements in this paper are full-step results for linear systems on fixed periodic grids with exact stage solutions. The constants may depend on the discrete grid and time step, and no mesh-uniform or nonlinear AP result has been established. The central high-order flux is intended for smooth periodic data and has no monotonicity guarantee; extending the analysis to nonlinear reconstructions, monotone AP fluxes, or parameter-robust inexact solves remains future work.
\end{newrevision}

\section*{Funding}
This work was supported by the Henan Provincial Natural Science Foundation for Young Scholars (Grant No. 252300423500), the Key Scientific Research Project of Colleges and Universities in Henan Province (Grant No. 26A110007), and the Doctoral Fund of Henan Polytechnic University (Grant No. B2024-60).

\end{document}